\documentclass[11pt,a4paper]{article}
\usepackage[utf8]{inputenc}
\usepackage[T1]{fontenc}
\usepackage[english]{babel}
\usepackage[toc,page]{appendix}
\usepackage[toctitles]{titlesec}
\usepackage{ragged2e}
\usepackage{caption}
\usepackage{xcolor}

\usepackage{graphicx}
\graphicspath{ {./images/} }

\usepackage{amsthm}
\usepackage{float}
\usepackage{enumitem}
\usepackage{amsmath}
\usepackage{amssymb}
\usepackage{mathrsfs}
\usepackage{txfonts}
\usepackage{mathtools}
\usepackage[perpage]{footmisc}
\usepackage{comment}
\usepackage{indentfirst}

\usepackage[a4paper,
 left=30mm,right=30mm,top=20mm,bottom=25mm,
 includeheadfoot,heightrounded]{geometry}
\usepackage{fancyhdr}
\usepackage{bbm}
\usepackage{hyperref}
\usepackage{url}

\newtheorem{definition}{Definition}[section]
\newtheorem{theorem}[definition]{Theorem}
\newtheorem{proposition}[definition]{Proposition}

\newtheorem{corollary}[definition]{Corollary}

\theoremstyle{definition}

\newtheorem{remark}[definition]{Remark}

\newtheorem{conjecture}[definition]{Conjecture}

\newtheorem{hypothesis}{Hypothesis}

\title{Stochastic persistence and extinction for degenerate stochastic Rosenzweig-MacArthur model}
\author{Michel Benaïm$^1$, Jérémy Colombo$^1$, Edouard Strickler$^2$ \\ 
\small $^1$ Institut de Mathématiques, Université de Neuchâtel, Switzerland\\
\small $^2$ IECL, Université de Lorraine, CNRS, Inria, Nancy, France.}
\date{\today}

\begin{document}

\maketitle

\justifying

\begin{abstract}
 \justifying{

 We consider the classical two-dimensional Rosenzweig–MacArthur prey–predator model \cite{Rosenzweig1963_Model} with a degenerate noise, whereby only the prey variable is subject to small environmental fluctuations. This model has already been introduced in \cite{Benaim2018_Persistence} and partially investigated by exhibiting conditions ensuring persistence.
 
 In this paper, we extend the results to study the conditions for persistence, the uniqueness of an invariant probability measure supported on the interior of $\mathbb R^2_+$ with a smooth density, and convergence in Total variation at a polynomial rate.

 Our contribution lies in providing a convergence rate in the case of persistence, as well as detailing the situations involving the extinction of one or both species. We also specify all the proofs of the intermediary results supporting our conclusions that are lacking in \cite{Benaim2018_Persistence}.
 }

\end{abstract}


\noindent\textbf{\href{https://mathscinet.ams.org/mathscinet/msc/msc2020.html}{MSC2020 Subject classifications}:} Primary 60J60; Secondary 37H15, 92D25, 60J35, 47D07, 37H30.

\tableofcontents

\section{Introduction}
We consider the classical two-dimensional Rosenzweig–MacArthur prey–predator model, where $x_1$ (respectively $x_2$) denotes prey density (respectively predator density), in which only the prey variable $x_1$ is subject to environmental fluctuations modeled by a Brownian motion. 

 This assumption reflects the fact that prey growth is highly dependent on random environmental factors, whereas predator dynamics, which are slower and dependent on prey, remain deterministic. Prey can be seen as small animals or plants, reacting more directly and quickly to environmental changes than predators, so we assume they are negligibly affected by sources of ambient randomness. In fact, one can see the prey as the resources of the predator. Other examples of dynamic population models of this type can be found in \cite{PDMP_Antoine} or \cite{Benaim_Antoine_2022}.

Namely, we consider the degenerate SDE system defined on $\mathbb R^2_+$ by
\begin{equation}\label{eq:rosenzweig-MacArthur-model-in-details}
 \begin{cases}
  \mathrm dx_1 = x_1 \left( \left(1 - \frac{x_1}{\kappa} - \frac{x_2}{1+x_1}\right) \mathrm dt + \varepsilon\,\mathrm d B_t\right), \\
  \mathrm dx_2 = x_2 \left(-\alpha +\frac{x_1}{1 + x_1}\right) \mathrm dt,
 \end{cases}
\end{equation}
  
where $\kappa$, $\varepsilon>0$, $0<\alpha< 1$, and $(B_t)_{t\geq0}$ is a standard one-dimensional Brownian motion. Remark that if $\alpha > 1$, $x_2(t)$ goes to $0$ almost surely as $t\to\infty$ since $\dot x_2(t)\leq x_2(t)(-\alpha +1)$.

In case of absence of noise, i.e. $\varepsilon = 0$, the behavior of (\ref{eq:rosenzweig-MacArthur-model-in-details}) is well-known (see e.g. \cite{smith2008rosenzweig}): 
\begin{itemize}
    \item Every trajectory starting from $x_1=0$ converges to the origin.
    \item If $\alpha \geq \frac{\kappa}{1 + \kappa}$, every solution starting from $\mathbb{R}_+^2 \setminus \{(0,0)\}$ converges to $(\kappa,0)$. In other words, predators go extinct while preys stabilize to a density equilibrium given by $\kappa$.
    \item If $\alpha<\frac{\kappa}{\kappa+1}$, the ODE admits a unique positive equilibrium $p=\left(\frac{\alpha}{1-\alpha}, \frac{\kappa-(\kappa+1)\alpha}{\kappa(1-\alpha)^2}\right).$  In addition:
\begin{itemize}
    \item 
     If $\alpha<\frac{\kappa-1}{1+\kappa}$, $p$ is a source and there exists a limit cycle $\gamma\subset \text{int}(\mathbb R_+^2):=\mathbb R_+^*\times \mathbb R_+^*$ surrounding $p$ and every solution starting from int$(\mathbb R_+^2)\setminus \{p\}$ has $\gamma$ as its limit set. 
    \item Otherwise, if $\frac{\kappa}{\kappa+1} > \alpha\geq\frac{\kappa-1}{1+\kappa}$, every solution starting from int$(\mathbb R_+^2)$ converges to $p$.
\end{itemize}
\end{itemize}
For $0 < \varepsilon^2 < 2$, define $k:= \frac{2}{\varepsilon^2} - 1 > 0$, $\theta := \frac{\varepsilon^2 \kappa}{2} < \kappa$, and $$\gamma_{\varepsilon, \kappa}(x) = \frac{x^{k-1} e^{-x/\theta}}{\Gamma(k)\, \theta^k}, \quad x \geq 0.$$

This is the density of a $\Gamma$-distribution with parameters $k, \theta$ and it represents the stationary distribution of the prey population in the absence of predators when $0<\varepsilon^2<2,$ as it will be demonstrated in Proposition \ref{prop:1-dim-logistic-results}\textit{\textbf{(ii)}}. We also define the long-term growth rate of the predator population as $$\Lambda(\varepsilon, \alpha, \kappa) = \int_0^{+\infty} \frac{x}{1 + x} \gamma_{\varepsilon, \kappa}(x)\, \mathrm dx - \alpha,$$ which represents the average per-capita growth rate of predators under the stationary prey distribution.

\subsection{Purpose of the article}

Our interest in the Rosenzweig-MacArthur model (\ref{eq:rosenzweig-MacArthur-model-in-details}) stems from the fact that it represents a degenerate SDE evolving on a non-compact state space. We will outline the conditions ensuring that degenerate models of the form of (\ref{eq:rosenzweig-MacArthur-model-in-details}) reflect typical ergodic properties. Specifically for the Rosenzweig-MacArthur model (\ref{eq:rosenzweig-MacArthur-model-in-details}), main properties can be summarized as follows:
\begin{enumerate}[label=(\textbf{\roman*})]
    \item If $0<\varepsilon^2<2$ and $\Lambda(\varepsilon, \alpha, \kappa)>0$ , $(x_1(t), x_2(t))_{t\geq 0}$ is \emph{stochastically persistent} in the sense that the law of $(X_t^x)_{t\geq 0}$ converges to a unique invariant probability measure, supported on the interior of $\mathbb R_+^2$ and with a smooth density with respect to the Lebesgue measure, whenever $x=(x_1(0), x_2(0))\in\text{int}(\mathbb R_+^2)$ is the initial condition of the system.
    \item If $0<\varepsilon^2<2$ and $\Lambda(\varepsilon, \alpha, \kappa)<0$, then $(x_2(t))_{t\geq 0}$ converges almost surely to $0$.
    \item If $\varepsilon^2>2$, then both $x_1(t)$ and $x_2(t)$ converge almost surely to $0$.
\end{enumerate}

The specific case $\Lambda(\varepsilon, \alpha, \kappa)=0$ has been investigated in \cite{doi:10.1137/20M131134X} as a critical case for (\ref{eq:rosenzweig-MacArthur-model-in-details}) where the authors proved that the predators go extinct in average, in the sense that, for all initial condition, almost surely, $$\lim_{t \to \infty} \frac{1}{t}\int_0^t x_2(s) ds = 0.$$

Note that the degenerate case studied here does not introduce new biological behaviours: it only makes the analysis technically more delicate as we will see later in this article. The fully non–degenerate setting is actually simpler, since hypoellipticity holds automatically.

\subsection{Outline of the article}

The structure of this article is as follows: Section \ref{section:results} summarizes the main results of this paper about persistence and extinction regarding (\ref{eq:rosenzweig-MacArthur-model-in-details}).

Section \ref{section:preliminaries} introduces the notation, establishes well-posedness of (\ref{eq:rosenzweig-MacArthur-model-in-details}) and recalls the notion of (extended) generators and carré du champ operators. We also detail our main standing assumption, Hypothesis \ref{hyp:main-hyp} and some of its consequences.

Section \ref{section:persistence} covers the notion of stochastic persistence and the way to use tools such as $H-$exponents in the case of (\ref{eq:rosenzweig-MacArthur-model-in-details}). Section \ref{section:convergence} introduces the relation with the almost sure convergence  (respectively in Total variation) of the empirical occupation measures (respectively the distribution) to the unique invariant probability measure on $\text{int}(\mathbb R_+^2)$. 

Section \ref{section:in-practice} presents how to verify stochastic persistence and extinction for the one-dimensional logistic SDE: it will support the proof of the extinction of one species in (\ref{eq:rosenzweig-MacArthur-model-in-details}). We also prove the extinction of both species and the persistence of (\ref{eq:rosenzweig-MacArthur-model-in-details}).

Section \ref{section:convergence-rate} details how to achieve a polynomial convergence in Total variation, in the specific case of non-compact extinction sets under a stronger Hypothesis \ref{hyp:main-hyp-strong}. We also explain why an exponential rate of convergence is out of reach in our framework.

For readability, some proofs as well as the intermediary and useful results are postponed in Appendix \ref{section:appendix}.

\section{Main results}\label{section:results}

Let $M:=\mathbb R_+^2$ be the state space of (\ref{eq:rosenzweig-MacArthur-model-in-details}). By standard results about finite-dimensional SDE with locally Lipschitz parameters, for each $x\in M$, there exists a unique strong solution $(X_t^x)_{t\geq 0}\subset M$ to (\ref{eq:rosenzweig-MacArthur-model-in-details}) with $X_0^x=x$ (see Proposition \ref{prop:hyp-LU-implies}) with Markov semigroup $(P_t)_{t\geq 0}$ defined for all measurable bounded functions $f$ as $$P_tf(x)=\mathbb E(f(X^x_t)), \quad \forall t\geq 0, \ x\in M.$$

For $I\subset\{1,2\}$, let $$M_0^I=\left\{x\in M: \prod_{i\in I}x_i=0\right\},$$ describes our \emph{extinction sets} while $$M_+^I:=\mathbb R_+^{2}\setminus M_0^I=\{x\in M: x_i>0, \forall i\in I\},$$ are our \emph{non-extinction sets}. If $\mathcal M$ denotes any of those subsets of $ M$, we remark that $\mathcal M$ is \emph{invariant} under $(P_t)_{t\geq 0}$ in the sense that $$P_t\mathbf{1}_{\mathcal M}=\mathbf{1}_{\mathcal M}, \quad \forall t\geq 0,$$ see Remark \ref{rem:invariant-extinction-set}. To shorten notations, we set $M_0^{(1,2)}=M_0$, $M_+^{(1,2)}=M_+$, $M^{\emptyset}_0=\emptyset$, and $M^{\emptyset}_+=M$.

Let $\mathcal P(M)$ be the space of probability measures on $( M,\mathcal B( M))$ where $\mathcal{B}(M)$ stands for the set of Borel subsets of $M$. A probability measure $\mu\in \mathcal P(M)$ is said to be invariant under $(P_t)_{t\geq 0}$ if $\mu P_t=\mu$, for all $t\geq 0$.

Recall that a sequence $(\mu_n)_{n\geq 0}\subset\mathcal P(M)$ is said to \emph{converge weakly} to $\mu\in\mathcal P(M)$ if $$\lim_{n\to\infty }\mu_n f=\mu f, \quad \forall f\in C_b( M),$$ where $C_b( M)$ is the set of continuous, bounded functions of $ M$ and we write $\mu_n\Rightarrow\mu$. Also, the \emph{Total variation distance} between $\alpha,\beta\in\mathcal{P}(M)$ is defined by $$\lVert\alpha-\beta\rVert_{\mathrm{TV}}:=\sup_{\|f\|_{\infty}\leq 1}|\alpha f-\beta f|,$$
where the $\sup$ is taken over measurable bounded functions $f$. Define the \emph{empirical occupation measures} $(\Pi^{x}_{t})_{t\geq 0}$ of the process $(X^{x}_{t})_{t\geq 0}$ by
\begin{equation}\label{eq:empirical-measures-definition-first}
 \Pi^{x}_{t}(B)=\frac{1}{t}\int_{0}^{t}\mathbf{1}_{\{X^{x}_{s}\in B\}}\,d s,\quad B\in\mathcal{B}(M).
\end{equation}

Thus $\Pi^{x}_{t}(B)$ is the fraction of time the process spends in $B$ up to time $t$. \\

Now, we establish a direct condition on $\Lambda(\varepsilon, \alpha, \kappa)$ to ensure the persistence of the model regarding $M_+$. Figure \ref{fig:stoch-r-m-model-persistence}, obtained in Python, illustrates the behavior of the process when $\Lambda(\varepsilon, \alpha, \kappa) > 0$ and $0<\varepsilon^2<2$.
\begin{figure}[H]
 \centering
 \includegraphics[width=0.5\textwidth]{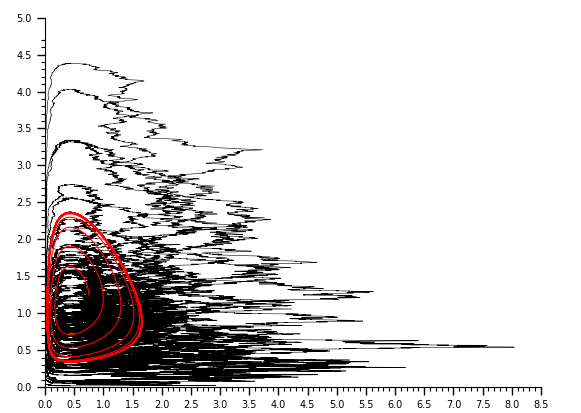}
 \caption{Simulation of (\ref{eq:rosenzweig-MacArthur-model-in-details}) starting at $(x_1(0),x_2(0))=(0.75,1.25)$ in persistence case with $\Lambda(\varepsilon=0.6,\ \alpha=0.3, \ \kappa=2.5) \approx 0.34 > 0$. The red trajectory is a trajectory of the deterministic system (i.e. for $\varepsilon=0$) while the black one describes the system (\ref{eq:rosenzweig-MacArthur-model-in-details}) with an Euler–Maruyama scheme.}
 \label{fig:stoch-r-m-model-persistence}
\end{figure}

\begin{theorem}[\textbf{Persistence}]\label{thm:rosenzweig-MacArthur-model-persistence}
    Suppose that $0<\varepsilon^2<2$ and $\Lambda(\varepsilon, \alpha, \kappa) > 0$. Then, there exists a unique invariant probability measure $\Pi$ on $ M_+$ such that, for all initial condition $x\in M_+$:
    \begin{enumerate}[label=(\textbf{\roman*})]
        \item $(\Pi_t^x)_{t\geq 0} \Rightarrow \Pi$, almost surely.
        \item For all $f\in L^1(\Pi)$ such that $\int_0^T f(X_s^x)ds<\infty$ for all $T>0$, $$\Pi_t^x f \underset{t\to\infty}{\longrightarrow} \Pi f, \quad \text{almost surely}.$$
       Moreover, for all $\theta <\frac{2}{\kappa \varepsilon^2}$, $(x_1,x_2)\mapsto e^{\theta(x_1+x_2)}$ lies in $L^1(\Pi)$, i.e. $\int_{ M_+}e^{\theta(x_1+x_2)}\Pi(\mathrm dx_1\mathrm dx_2)<\infty.$
        \item $(P_t(x,\cdot))_{t\geq0}$ converges in Total variation towards $\Pi$ at a polynomial rate, in the sense that there exists $\lambda >0$ such that $$\lim_{t\to\infty}t^{\lambda}||P_t(x,\cdot)-\Pi(\cdot)||_{TV}=0.$$
        \item $\Pi$ has a smooth density (with respect to the Lebesgue measure), strictly positive, on $ M_+$.
\end{enumerate}
\end{theorem}

Its proof is detailed in Section \ref{section:appendix-mainresults}.\\

Then, we are interested in the extinction situation. We will see that the one-dimensional logistic SDE controls the behavior of the prey and we prove the extinction of one or both species with a comparison theorem for SDEs.

In particular, if the condition $\Lambda(\varepsilon, \alpha, \kappa) > 0$ is not respected, we show the almost-sure convergence of $(X_t^x)_{t\geq 0}$ to $M_0^{(2)}:=\{x\in M:x_2=0\}$.

\begin{theorem}[\textbf{Extinction of species $x_2$}]\label{thm:rosenzweig-MacArthur-model-extinction-1}
    If $0<\varepsilon^2<2$ and $\Lambda(\varepsilon, \alpha, \kappa) < 0$, then $\forall x\in M$, $x_2^x(t)\underset{t\to\infty}{\longrightarrow} 0$, $\mathbb P-$almost surely and exponentially fast. More precisely, $$\limsup_{t\to\infty} \frac{1}{t}\log\big(x_2(t)\big)\leq\Lambda(\varepsilon, \alpha, \kappa).$$
\end{theorem}

Figure \ref{fig:stoch-r-m-model-extinction-1} illustrates this typical extinction behavior.
\begin{figure}[H]
    \centering
    \includegraphics[width=0.5\textwidth]{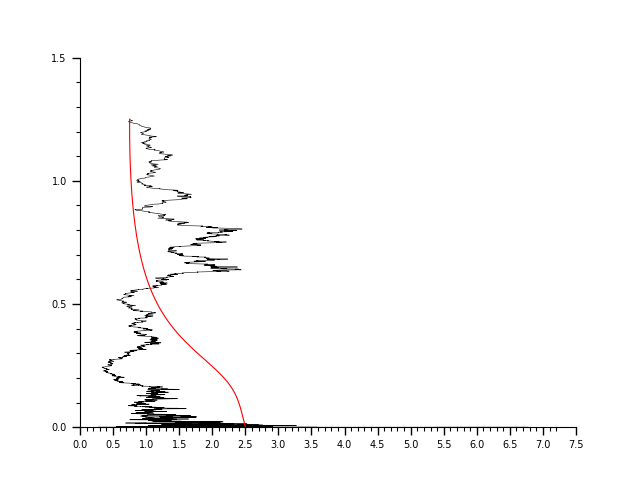}
    \caption{Simulation of (\ref{eq:rosenzweig-MacArthur-model-in-details}) starting at $(x_1(0),x_2(0))=(0.75,1.25)$ in extinction case with $\Lambda(\varepsilon=0.6,\ \alpha=0.9, \ \kappa=2.5) \approx -0.26 < 0$. The red trajectory is a trajectory of the deterministic system (i.e. for $\varepsilon=0$) while the black one describes the system (\ref{eq:rosenzweig-MacArthur-model-in-details}) with an Euler–Maruyama scheme.}
    \label{fig:stoch-r-m-model-extinction-1}
\end{figure}

Figure \ref{fig:stoch-r-m-model-extinction-1-global-zoom} illustrates the situation where $0<\varepsilon^2<2$ and $\Lambda(\varepsilon,\alpha, \kappa)<0$ while $\kappa$ is chosen large enough so that the deterministic trajectory converges to a limit cycle.
\begin{figure}[H]
    \centering
    \includegraphics[width=0.35\textwidth]{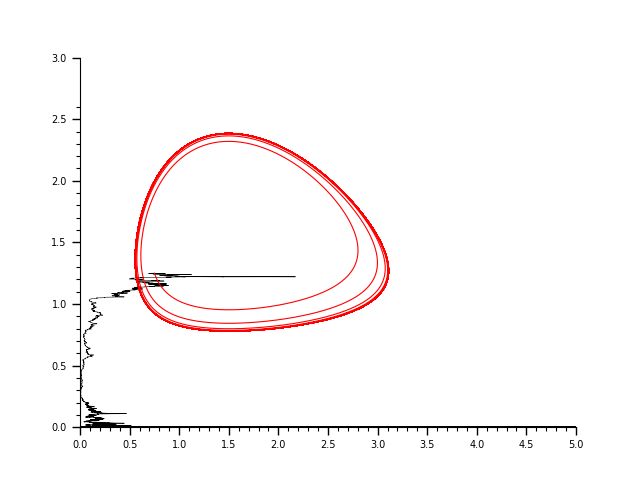}
    \includegraphics[width=0.35\textwidth]{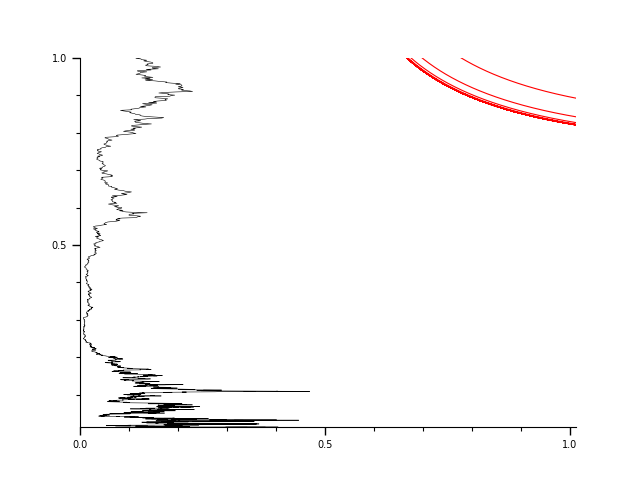}
    \caption{Left figure is a simulation of (\ref{eq:rosenzweig-MacArthur-model-in-details}) starting at $(x_1(0),x_2(0))=(0.75,1.25)$ in the case of the extinction of $x_2$ only with $\Lambda(\varepsilon=1.35,\ \alpha=0.6, \ \kappa=4.5) \approx -0.48 < 0$. The red trajectory is a trajectory of the deterministic system (i.e. for $\varepsilon=0$) while the black one describes the system (\ref{eq:rosenzweig-MacArthur-model-in-details}) with an Euler–Maruyama scheme. Right figure is a close-up of the situation near the extinction set $x_2=0$ while $x_1$ does not reach $0$.}
    \label{fig:stoch-r-m-model-extinction-1-global-zoom}
\end{figure}

Observe that the case $\alpha \geq 1$, ruled out in the introduction, naturally implies $\Lambda(\varepsilon, \alpha, \kappa) \leq 0$ and is thus covered by Theorem \ref{thm:rosenzweig-MacArthur-model-extinction-1} or by the critical case studied in \cite{doi:10.1137/20M131134X}.

\begin{proposition}\label{prop:convergence-when-x-2-extinction}
    Under the assumption from Theorem \ref{thm:rosenzweig-MacArthur-model-extinction-1}, it yields $$\Pi_t^x\Rightarrow \gamma_{\varepsilon, \kappa}(x_1)\mathrm dx_1 \otimes \delta_0(x_2)\mathrm dx_2, \quad \text{almost surely},$$ for all $x\in M_+^{(1)}=\{x\in M:x_1> 0\}$.
\end{proposition}

Despite this result, there is no evidence that the law of $x_1(t)$ converges to $\gamma_{\varepsilon,\kappa}$.
\begin{conjecture}\label{conj:rosenzweig-MacArthur-model-extinction-1-convergence-x1}
 \textit{In the setting of Theorem \ref{thm:rosenzweig-MacArthur-model-extinction-1}, $\forall x\in  M_+$, the law of $x_1^x(t)$ converges weakly to $\gamma_{\varepsilon,\kappa}(\mathrm dx)$.}
\end{conjecture}

\begin{figure}[H]
 \centering
 \includegraphics[width=0.5\textwidth]{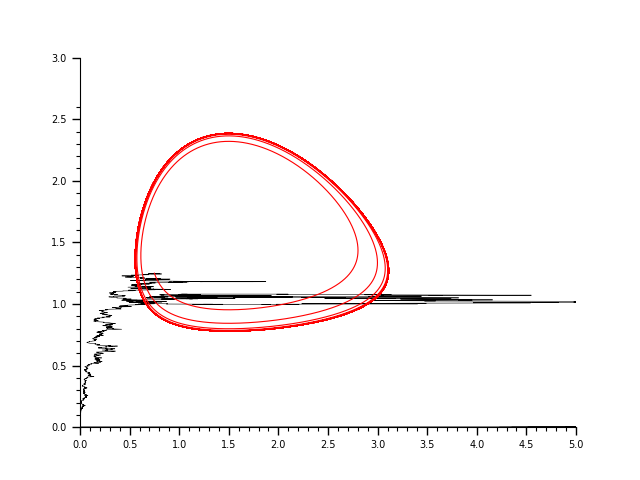}
 \caption{Simulation of (\ref{eq:rosenzweig-MacArthur-model-in-details}) starting at $(x_1(0),x_2(0))=(0.75,1.25)$ in global extinction case with $\Lambda(\varepsilon=1.5,\ \alpha=0.6, \ \kappa=4.5) \approx -0.79 < 0$. The red trajectory is a trajectory of the deterministic system (i.e. for $\varepsilon=0$) while the black one describes the system (\ref{eq:rosenzweig-MacArthur-model-in-details}) with an Euler–Maruyama scheme.}
 \label{fig:stoch-r-m-model-extinction-2}
\end{figure}

Finally, we focus on the case where the condition $\varepsilon^2<2$ is not respected: in this situation where the environmental fluctuation is too large, it leads to the extinction of both species as depicted in Figure \ref{fig:stoch-r-m-model-extinction-2}.

\begin{theorem}[\textbf{Extinction of both species}]\label{thm:rosenzweig-MacArthur-model-extinction-2}
 If $\varepsilon^2>2$, then $\forall x\in M$, $X_t^x\underset{t\to\infty}{\longrightarrow} (0,0)$, $\mathbb P-$almost surely and exponentially fast. More precisely, $$\limsup_{t\to\infty} \frac{1}{t}\log\big(x_1(t)\big)\leq1-\frac{\varepsilon^2}{2}, \ \mathrm{ and } \ \limsup_{t\to\infty} \frac{1}{t}\log\big(x_2(t)\big)\leq-\alpha.$$
\end{theorem}

The proofs of both Theorems \ref{thm:rosenzweig-MacArthur-model-extinction-1} and \ref{thm:rosenzweig-MacArthur-model-extinction-2} as well as Proposition \ref{prop:convergence-when-x-2-extinction} are detailed in Section \ref{section:extinction-proof}.\\

As a summary of this situation, Figure \ref{fig:Lambda-estimation} depicts the different situation we face by evaluating $\Lambda(\varepsilon, \alpha, \kappa)$ with $\alpha=0.5$ fixed while varying $\varepsilon$ and $\kappa$.
\begin{figure}[H]
 \centering
 \includegraphics[width=0.7\textwidth]{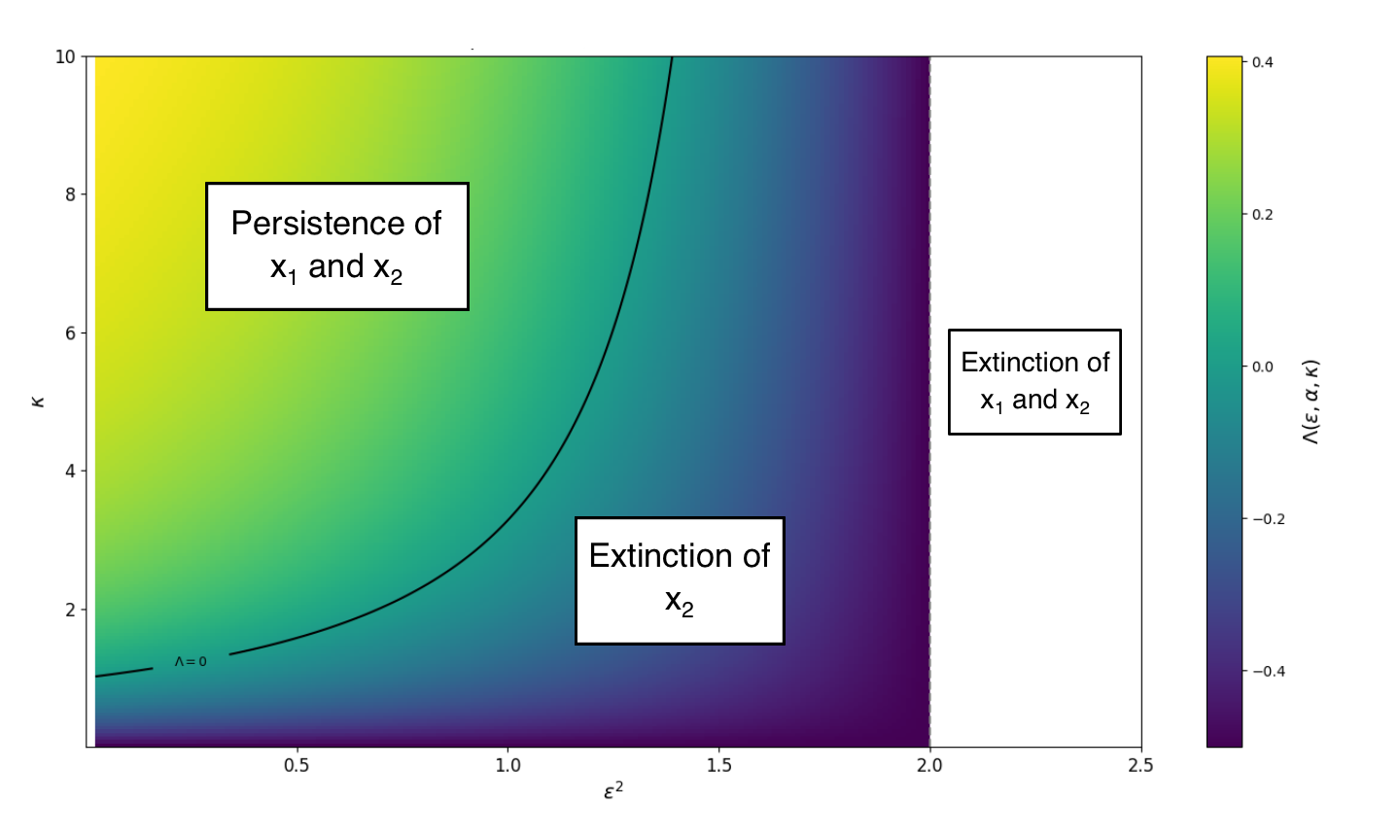}
 \caption{Evaluation of $\Lambda(\varepsilon,\ \alpha, \ \kappa)$ for $\alpha=0.5$ fixed. The different zones of the graph detail when we are in a persistence situation (above the $\Lambda=0$ curve), general extinction (when $\varepsilon^2>2$) and extinction of only $x_2$ (below the $\Lambda=0$ and when $\varepsilon^2<2$).}
 \label{fig:Lambda-estimation}
\end{figure}

\begin{remark}
    Since $\gamma_{\varepsilon,\kappa}$ is a $\Gamma-$ distribution whose expectation is $k\theta = \kappa(1 - \frac{\varepsilon^2}{2})$ and variance $k\theta^2 = \frac{\kappa^2 \varepsilon^2}{2}(1 - \frac{\varepsilon^2}{2})$, if $X\sim \gamma_{\varepsilon,\kappa}$, then $\lim_{\varepsilon\to 0}\mathbb E(X)=\kappa$ and $\lim_{\varepsilon\to 0}\mathrm{Var}(X)=0.$ By Bienaymé–Tchebyschev, for any $\delta>0$, $$\mathbb P(|X-\mathbb E(X)|>\delta)\leq \frac{\mathrm{Var}(X)}{\delta ^2}\underset{\varepsilon\to0}{\longrightarrow}0.$$ Thus, the law of $X$ converges to $\delta_{\kappa}$ and in our context, $$\Lambda(\varepsilon,\alpha,\kappa)\underset{\varepsilon\to0}{\longrightarrow} \frac{\kappa}{1+\kappa}-\alpha = \Lambda(0,\alpha, \kappa),$$
    which is deterministic persistence threshold. Similarly, since $\lim_{\varepsilon\to\sqrt{2}} \mathbb E(X)=\lim_{\varepsilon\to\sqrt{2}} \mathrm{Var}(X)=0,$ it yields $$\Lambda(\varepsilon,\alpha,\kappa)\underset{\varepsilon\to\sqrt{2}}{\longrightarrow}-\alpha=\Lambda(\sqrt{2},\alpha, \kappa).$$ 
\end{remark}

In the same spirit, Figure \ref{fig:Lambda-estimation-various} depicts the different situation with $\varepsilon=0.6$ fixed while varying $\alpha$ and $\kappa$.
\begin{figure}[H]
 \centering
 \includegraphics[width=0.7\textwidth]{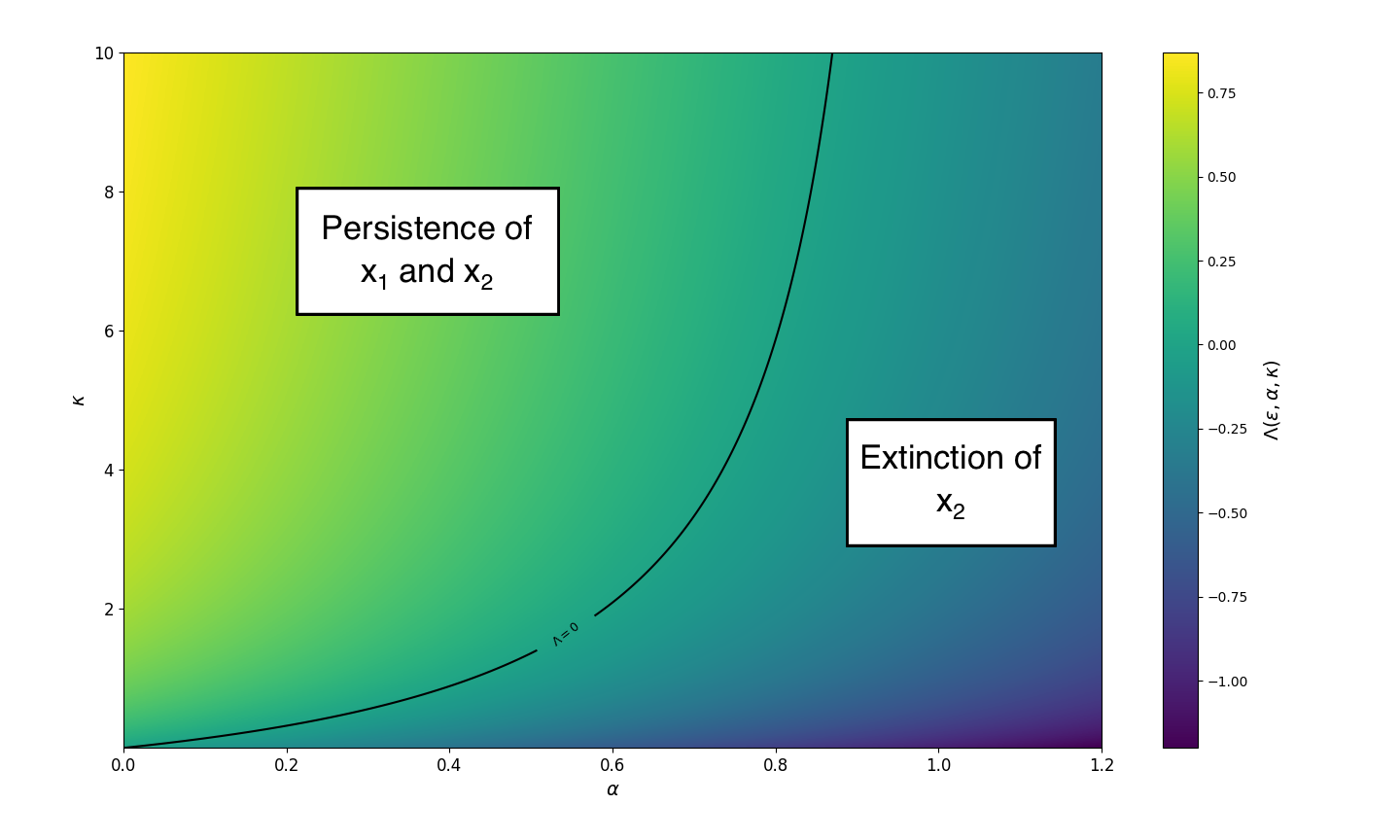}
 \caption{Evaluation of $\Lambda(\varepsilon,\ \alpha, \ \kappa)$ for $\varepsilon=0.6$ fixed. The different zones of the graph detail when we are in a persistence situation (above the $\Lambda=0$ curve) and extinction of only $x_2$ (below the $\Lambda=0$).}
 \label{fig:Lambda-estimation-various}
\end{figure}

\section{Notations and preliminaries}\label{section:preliminaries}

We focus on \emph{Kolmogorov stochastic differential equations}, which describes a system of SDEs on the state space $M:= \mathbb R_+^n$ of the form
\begin{equation}\label{eq:kolmogorov}
 \mathrm dx_i = x_i\Bigl[F_i(x)\,\mathrm dt + \sum_{j=1}^{m}\Sigma^{j}_{i}(x)\,d B^{j}_{t}\Bigr], \quad i=1,\dots,n,
\end{equation}

where $F_i$, $\Sigma_i^j$ are real valued locally Lipschitz maps on $M$, $F_i$ represents the per-capita growth rates of the species in absence of noise, and $(B_t^1, \cdots, B_t^m)_{t\geq 0}$ is an $m-$dimensional standard Brownian motion that models the environmental noise affecting the growth rates. 

For simplicity, we assume that $\Sigma_i^j$ is bounded: this assumption can be relaxed under other conditions such as described in \cite{Hening2018_Coexistence}. Note that the similar models investigated by these authors are limited to non-degenerate situations. 

We let $a(x)$ denote the positive semi‑definite matrix defined by
\begin{equation}
 a_{ij}(x)=\sum_{k=1}^{m}\Sigma^{k}_{i}(x)\,\Sigma^{k}_{j}(x).
\end{equation}

For all twice continuously differentiable functions $f\colon M\to\mathbb{R}$, let
\begin{equation}\label{eq:generator-kolmogorov-general}
 Lf(x)=\sum_{i=1}^{n}x_iF_i(x)\frac{\partial f}{\partial x_i}(x)+\frac12\sum_{i,j=1}^{n}x_i x_j a_{ij}(x)\frac{\partial^{2}f}{\partial x_i\partial x_j}(x),
\end{equation}
and
\begin{equation}\label{eq:extended-carre-du-champ-kolmogorov-general}
 \Gamma (f)(x)=\sum_{i,j=1}^{n}x_i x_j a_{ij}(x)\frac{\partial f}{\partial x_i}(x)\frac{\partial f}{\partial x_j}(x).
\end{equation}

Recall that continuous function $W\colon M\to\mathbb{R}$ is \emph{proper} if the sublevel sets $\{x\in M:W(x)\leq R\}$ are compact for all $R>0$, or equivalently $\lim_{||x||\to\infty} W(x)= \infty$.

The following hypothesis is our \emph{standing assumption} that will be assumed to be verified along the text.
\begin{hypothesis}[\textbf{Standing assumption}]\label{hyp:main-hyp}
 There exist a $C^{2}$ proper map $U\colon M\to[1,\infty)$ and constants $a>0$, $b\geq 0$ such that
\begin{equation}\label{eq:hyp-LU}
 LU\leq -a U+b.
\end{equation}
\end{hypothesis}

\begin{proposition}\label{prop:hyp-LU-implies}
 Under Hypothesis \ref{hyp:main-hyp}, for each $x\in M$, there exists a unique (strong) solution $(X^{x}_{t})_{t\geq 0}\subset M$ to (\ref{eq:kolmogorov}) with $X^{x}_{0}=x$, and $X^{x}_{t}$ is continuous in $(t,x)$.
\end{proposition}

We postpone the proof inspired by Proposition 3.2 in \cite{Benaim2018_Persistence} to Appendix \ref{appendix-proof-hyp-LU-implies}. For each $f\in \mathcal B_b( M)$ bounded measurable functions on $ M$, we set $$P_tf(x)=\mathbb E(f(X_t^x)), \quad \forall t\geq 0, \ x\in M.$$ 

The family $(P_t)_{t\geq 0}$ is a Markov semigroup satisfying the \emph{Markov property}, which is
\begin{equation*}
 \mathbb{E}\bigl[f(X^{x}_{t+s})\,\vert\,\mathcal{F}_{t}\bigr]=(P_{s}f)(X^{x}_{t}), \quad \forall x \in M, \ \mathbb{P}\text{-a.s.}
\end{equation*}

According to Proposition \ref{prop:hyp-LU-implies}, $(P_t)_{t\geq 0}$ is a $C_b( M)-$Feller semigroup in the sense that $$P_t(C_b( M))\subset C_b( M), \quad \forall t\geq 0,$$ and 
\begin{equation*}
 \lim_{t\downarrow 0}P_tf(x)=f(x), \quad \forall f\in C_b( M), \ x\in M.
\end{equation*}

As before, for $I\subset\{1,\cdots,n\}$, let
\begin{equation}\label{eq:definition-M-0-I}
    M_0^I:=\left\{x\in M: \prod_{i\in I} x_i=0, \right\},
\end{equation} 
describes our \emph{extinction sets} while 
\begin{equation}\label{eq:definition-M-+-I}
M_+^I:=\{x\in M: x_i>0, \forall i\in I\},
\end{equation}
is our \emph{non-extinction set}. In particular, if $\mathcal M$ denotes any of those subsets, given the form of (\ref{eq:kolmogorov}), they are invariant under $(P_t)_{t\geq 0}$ (see Remark \ref{rem:invariant-extinction-set}). For simplicity, let $M_0^{(1,\cdots, n)}=M_0$ and $M_+^{(1,\cdots, n)}=M_+$.

Recall that a probability measure $\mu\in\mathcal{P}( M)$ is called \emph{invariant} if
$$\mu P_{t}=\mu,\quad t\geq 0,$$
that is, $\mu(P_{t}f)=\mu f$ for all $f\in\mathcal{B}( M)$ (or $C_{b}( M)$) and all $t\geq 0$. Denote the set of invariant probability measures of $(P_{t})_{t\geq 0}$ by $\mathcal{P}_{\mathrm{inv}}( M)$. We also define
\begin{align*}
 \mathcal{P}_{\mathrm{inv}}(\mathcal M)&=\bigl\{\mu\in\mathcal{P}_{\mathrm{inv}}( M):\mu(\mathcal M)=1\bigr\}.
\end{align*}

Recall that the empirical occupation measures $(\Pi^{x}_{t})_{t\geq 0}$ of the process $(X^{x}_{t})_{t\geq 0}$ is defined by (\ref{eq:empirical-measures-definition-first}).

One key consequence of Hypothesis \ref{hyp:main-hyp} is the tightness of $(\Pi^{x}_{t})_{t\geq 0}$, as described in the following more general result.   Note that these conclusions do not depend on the specific form of the SDE (\ref{eq:kolmogorov}) and remain true for any diffusion on $\mathbb R_+^n$ with locally Lipschitz coefficients, provided that Hypothesis \ref{hyp:main-hyp} is verified.

\begin{proposition}\label{prop:conditions-hypotheses-hold}
 Under Hypothesis \ref{hyp:main-hyp}, the following properties hold:
 \begin{enumerate}[label=(\textbf{\roman*})]
  \item Let $a,b$ be the constants from (\ref{eq:hyp-LU}), then
  \begin{equation}\label{eq:hyp-3-b-implies-bounded-PW}
   P_{t} U\leq e^{-a t}\left(U-\frac{b}{a}\right)+\frac{b}{a}\leq e^{-a t} U+\frac{b}{a}, \quad \forall t\geq 0.
  \end{equation}
  \item For all $\Pi\in\mathcal P_{\mathrm{inv}}(M)$, $\Pi (U)\leq \frac{b}{a}$.
  \item For all $x\in M$,
  \begin{equation}\label{eq:pi-t-x-U-bounded}
      \limsup_{t\to\infty} \Pi_t^x\left(\sqrt U\right)\leq\left(\frac{2-\sqrt{e^{-a}}}{1-\sqrt{e^{-a}}}\right) \frac{\sqrt b}{\sqrt a}, \quad \text{almost surely}.
  \end{equation} In particular, $(\Pi^{x}_{t})_{t\geq 0}$ is almost surely tight, and every limit point lies in $\mathcal{P}_{\mathrm{inv}}(M)$.
 \end{enumerate}
\end{proposition}

\begin{remark}\label{rem:P-inv-empty}
A consequence of the tightness of $(\Pi^{x}_{t})_{t\geq 0}$ is that both $\mathcal{P}_{\mathrm{inv}}( M)$ and $\mathcal{P}_{\mathrm{inv}}(M_0)$ are non‐empty. However, $\mathcal{P}_{\mathrm{inv}}(M_+)$ may still be empty.
\end{remark}

Given the \emph{$1-$resolvent kernel} $G$ defined as
\begin{equation}\label{eq:1-dim-resolvent-kernel}
 Gf = \int_0^{\infty} e^{-t}P_t f \mathrm dt, \quad \forall f\in \mathcal B( M),
\end{equation}
a map $h\in \mathcal B( M)$ is called \emph{$(G,\mu)-$invariant} if $Gh = h,$ $\mu-$almost surely. Similarly, a measurable set $B\subset M$ is called $(G,\mu)-$invariant if $\mathbf{1}_B$ is $(G,\mu)-$invariant. An invariant probability measure $\mu\in \mathcal P_{inv}( M)$ is called \emph{ergodic} if every $(G,\mu)-$invariant map is $\mu-$almost surely constant. 

Denote the set of ergodic measures by $\mathcal P_{erg}( M)$. If $\mathcal M$ denotes any of the subsets of $M_0^I$ or $M_+^I$ as defined before, let $$\mathcal P_{erg}(\mathcal M)=\{\mu \in \mathcal P_{erg}( M): \mu(\mathcal M)=1\}.$$

Denote by $\mathcal D_e^{\mathcal M}$ the vector space called the \emph{extended domain} and $\mathcal L_e^{\mathcal M}:\mathcal D_e^{\mathcal M}\to C(\mathcal M)$ the linear map called the \emph{extended generator}. Then, $\mathcal D_e^{\mathcal M}$ is formed by continuous function $f:\mathcal M\to\mathbb R$, possibly unbounded, such that for all $x\in \mathcal M$, the process 
\begin{equation}\label{eq:extended-generator-implies}
 M^{f}_{t}(x)=f(X^{x}_{t})-f(x)-\int_{0}^{t}\mathcal{L}^{\mathcal M}_{\mathrm{e}}f\bigl(X^{x}_{s}\bigr)\,d s,\quad t\geq 0,
\end{equation}
is a $\left((\mathcal F_t\right)_{t\geq 0}\mathbb P_x)-$local martingale.

Another powerful tool is the \emph{extended carré du champ}.
\begin{definition}\label{def:carre-du-champ}
 Let $f\in\mathcal D_e^{\mathcal M}$ such that $f^2\in\mathcal D_e^{\mathcal M}$: we say that $f\in\mathcal D_e^{2,\mathcal M}$ on which we define the \emph{extended carré du champ} as $$\Gamma_e^{\mathcal M}f:=\mathcal L_e^{\mathcal M}f^2-2f\mathcal L_e^{\mathcal M}f.$$
\end{definition}

It is well-known that for $f\in \mathcal D_e^{2,\mathcal M}$ and for any $x\in\mathcal M$, the process $$(M^f_t)^2(x)-\int_0^t\Gamma_e^{\mathcal M}f(X_s^x)ds, \quad \forall t\geq 0,$$ is a $\left((\mathcal F_t\right)_{t\geq 0}\mathbb P_x)-$local martingale, and in particular $$\langle M^f(x) \rangle_t=\int_0^t \Gamma_e^{\mathcal M}(f)(X_s^x)ds,$$ where $\langle M^f(x) \rangle_t$ is the predictable quadratic variation of $(M_t^f(x))_{t\geq 0}$. For shortness, we let $$\mathcal D_e^{ M}=\mathcal D_e, \ \mathcal L_e^{ M}=\mathcal L_e, \ \mathcal D_e^{2, M}=\mathcal D_e^2, \ \Gamma_e^{ M} =\Gamma_e,$$ and $$\mathcal D_e^{M_+^I}=\mathcal D_e^I, \ \mathcal L_e^{M_+^I}=\mathcal L_e^I, \ \mathcal D_e^{2,M_+^I}=\mathcal D_e^{2,I}, \ \Gamma_e^{M_+^I}=\Gamma_e^I.$$ When $I=\{1,\cdots, n\}$, we simply replace the $I$ exponents by $+$.

A direct consequence of Proposition \ref{prop:hyp-LU-implies} and Itô's formula is the following relation.
\begin{proposition}\label{prop:domain-definition-c2}
 For all $f\in C^{2}( M)$, $f\in\mathcal{D}_{\mathrm{e}}^2$, $\mathcal L_{\mathrm{e}}f=L f$ and $\Gamma_e f = \Gamma  f$.
\end{proposition}

\begin{proof}
Let $f\in C^2( M)$, then the property $\mathcal{L}_{\mathrm{e}}(f)=Lf$ follows from Itô’s formula since the process $$f(X_t^x)-f(x)-\int_0^t Lf(X_s^x)ds= \sum_{i=1}^n\int_0^t\frac{\partial f}{\partial x_i}\bigl(X^x_s\bigr) \left[X^x_{s,i}\sum_{j=1}^m\Sigma^j_i\bigl(X^x_s\bigr)\right] dB^j_s,$$ is a local martingale by continuity of $t\mapsto X_t^x$ and the non-explosion of the solution. For $\Gamma_{\mathrm{e}}$, let $f\in C^2( M)$. Since
\begin{align*}
 \mathcal L_e(f^2)&=\sum_{i=1}^n x_iF_i(x)\frac{\partial f^2}{\partial x_i}(x)+\frac{1}{2}\sum_{i,j=1}^n x_ix_ja_{ij}(x)\frac{\partial^2 f^2}{\partial x_i\partial x_j}(x)\\
 &=\sum_{i=1}^n x_iF_i(x)2f(x)\frac{\partial f}{\partial x_i}(x)+\frac{1}{2}\sum_{i,j=1}^n x_ix_ja_{ij}(x)\left(2\frac{\partial f}{\partial x_i}(x)\frac{\partial f}{\partial x_j}(x)+2f(x)\frac{\partial^2 f}{\partial x_i\partial x_j}(x)\right),
\end{align*}

and $$2f\mathcal L_e(f)=\sum_{i=1}^n x_iF_i(x)2f(x)\frac{\partial f}{\partial x_i}(x)+\frac{1}{2}\sum_{i,j=1}^n x_ix_ja_{ij}(x)2f(x)\frac{\partial^2 f}{\partial x_i\partial x_j}(x),$$ it follows that $$\Gamma_e(f)(x) = \mathcal L_e(f^2)-2f\mathcal L_e(f) = \sum_{i,j=1}^n x_ix_ja_{ij}(x)\frac{\partial f}{\partial x_i}(x)\frac{\partial f}{\partial x_j}(x)=\Gamma (f)(x).$$
\end{proof}

\begin{definition}\label{def:strong-law}
 Let $f\in \mathcal{D}^{\mathcal M}_{\mathrm{e}}$. We say that $f$ satisfies the \emph{strong law} if, for every $x\in \mathcal M$,
 \begin{equation}\label{eq:strong-law}
  \lim_{t\to\infty}\frac{M^{f}_{t}(x)}{t}=0\quad\mathbb{P}^{x}\text{-a.s.}
 \end{equation}
\end{definition}

The following tool is a key result to ensure that the local martingale $(M_t^f(x))_{t\geq 0}$ defined in (\ref{eq:extended-generator-implies}) is a true martingale and in addition that $f$ satisfies the strong law.

\begin{proposition}\label{prop:extended-carre-du-champ-strong-law}
 Let $ f \in \mathcal{D}^{\mathcal M}_e $. Consider the following assertions:
 \begin{enumerate}[label=(\textbf{\roman*})]
  \item $ f \in \mathcal{D}^{2, \mathcal M}_e $, and for all $ x \in \mathcal M $, $$ \limsup_{t \to \infty} \frac{1}{t} \int_0^t P_s(\Gamma^{\mathcal M}_e(f))(x) ds < \infty;$$
  \item For all $ x \in \mathcal M $, $ (M^f_t(x))_{t \geq 0} $ is a true $ L^2 \ ((\mathcal{F}_t)_{t \geq 0}, \mathbb{P}_{x}) -$martingale, and
  \begin{equation}\label{eq:square-martingale}
   \limsup_{t \to \infty} \frac{\mathbb{E}^x \left[ (M^f_t(x))^2 \right]}{t} < \infty;
  \end{equation}
  \item $f$ satisfies the strong law.
\end{enumerate}
Then $$\textbf{(i)} \Rightarrow \textbf{(ii)} \Rightarrow \textbf{(iii)}.$$
\end{proposition}

The proof is postponed in Appendix. As a direct consequence, we obtain the following criteria to guarantee whenever the local martingale $(M_t^f)_{t\geq 0}$ is a true martingale satisfying the strong law.

\begin{corollary}\label{cor:hyp-3-implies}
Let $U$ be as in Hypothesis \ref{hyp:main-hyp} and $f \in \mathcal{D}^{2, \mathcal M}_e $ be such that \begin{equation}\label{eq:condition-Gamma-U-implies-strong-law}
    \Gamma^{\mathcal M}_e(f)(x) \leq \tilde{a} U(x) + \tilde{b},
\end{equation} for all $ x \in \mathcal M$ with $\tilde a>0$, $\tilde b\geq 0$. Then, for all $x\in\mathcal M$, $(M^f_t(x))_{t \geq 0} $ is a true $ L^2 $ martingale and $f$ satisfies the strong law.
\end{corollary}

\begin{proof}
By (\ref{eq:hyp-3-b-implies-bounded-PW}), $$\frac{1}{t}\int_0^t P_s U(x)ds\leq  \frac{U(x)}{at} + \frac{b}{a}.$$

Combined with (\ref{eq:condition-Gamma-U-implies-strong-law}), it follows that
\begin{equation*}
    \limsup_{t\to\infty}\frac{1}{t} \int_0^t P_s (\Gamma_e(f))(X_s^x)ds \leq \limsup_{t\to\infty} \frac{\tilde a}{t} \int_0^t P_sU(x)ds + \tilde b\leq \limsup_{t\to\infty}\frac{\tilde a U(x)}{at}+\frac{\tilde a b}{a} +\tilde b<\infty.
\end{equation*}
Thus, Proposition \ref{prop:extended-carre-du-champ-strong-law}\textit{\textbf{(i)}} holds so $(M^f_t(x))_{t \geq 0} $ is a true $ L^2$ martingale and $f$ satisfies the strong law.
\end{proof}

\section[Stochastic Persistence]{An Introduction to Stochastic Persistence}\label{section:persistence}

We focus on Kolmogorov SDEs of the form of (\ref{eq:kolmogorov}). Recall that Hypothesis \ref{hyp:main-hyp} is assumed to be true. From Remark \ref{rem:P-inv-empty}, we now want to ensure that $\mathcal{P}_{inv}(M_+)$ is non-empty. To this effect, we introduce the notion of \emph{stochastic persistence}.

The following definition, inspired by the seminal work of Chesson in \cite{Chesson1978_Predator}, \cite{Chesson1982_Stabilizing}, follows the presentation of Schreiber in \cite{Schreiber2012_Persistence}.

Recall that for any $I\subset\{1,\cdots, n\}$, $M_0^I$ denotes the extinction set of species $i\in I$, as in (\ref{eq:definition-M-0-I}), while $M_+^I$ is the non-extinction set of species $i\in I$, as in (\ref{eq:definition-M-+-I}).

\begin{definition}\label{def:stochastic-persistence}
 The family of processes $\{(X^{x}_{t})_{t\geq 0}:x\in M_+^I\}$ is called \emph{stochastically persistent} if for every $\varepsilon>0$ there exists a compact set $K_{\varepsilon}\subset M_+^I$ such that, for all $x\in M_+^I$, $$\mathbb{P}\Bigl(\liminf_{t\to\infty}\Pi^{x}_{t}(K_{\varepsilon})\geq 1-\varepsilon\Bigr)=1.$$
\end{definition}

Analyzing stochastic persistence requires control of the process near the extinction set $M_0^I$. In addition to the original Lyapunov-type condition in Hypothesis \ref{hyp:main-hyp}, we introduce the following control near extinction.

\begin{definition}\label{def:H-persistence}
 The family $\{(X^{x}_{t})_{t\geq 0}:x\in M_+^I\}$ is called \emph{$H$‑persistent} if there exist continuous maps $V\colon M_+^I\to\mathbb{R}_+$ and $H\colon M\to\mathbb{R}$ such that
 \begin{enumerate}[label=(\textbf{\roman*})]
  \item $V\in D_e^{I}$ and $\mathcal L_{\mathrm{e}}^I(V)=H|_{M_+^I}$ ;
  \item $V$ satisfies the strong law (Definition \ref{def:strong-law});
  \item $\frac{\sqrt{U}}{1+|H|}$ is proper, where $U$ is like in Hypothesis \ref{hyp:main-hyp} and $\sqrt{U}$ satisfies the strong law.
   \item For all $\mu\in\mathcal{P}_{\mathrm{erg}}(M_0^I)$, $\mu H<0$.
 \end{enumerate}
\end{definition}

Under \textit{\textbf{(iii)}}, $\sqrt{U}$ dominates $|H|$ outside a compact set so that $H\in L^{1}(\mu)$ for all $\mu\in\mathcal{P}_{\mathrm{inv}}( M)$ by Proposition \ref{prop:conditions-hypotheses-hold}\textit{\textbf{(ii)}} (see Remark \ref{rem:P-inv-empty}), and $\mu H$ above is well-defined. In fact, condition \textit{\textbf{(iv)}} is a shortcut for a broader tool defined as the \emph{$H-$exponents}.

\begin{definition}\label{def:H-exponents}
 Given $V$ and $H$ as in Definition \ref{def:H-persistence}, define the \emph{lower} and \emph{upper $H$–exponents} of $\{(X^x_{t})_{t\geq0}:x\in M_+^I\}$ by $$\Lambda^{-}(H):=-\sup \{\mu H : \mu\in\mathcal{P}_{\mathrm{erg}}(M_0^I)\},\ \Lambda^{+}(H):=-\inf\{\mu H: \mu\in\mathcal{P}_{\mathrm{erg}}(M_0^I)\}.$$

 In particular, condition \textit{\textbf{(iv)}} of Definition \ref{def:H-persistence} rewrites $\Lambda^-(H)>0.$
\end{definition}

\begin{remark}\label{rem:VH-domain}
The key point in Definition \ref{def:H-persistence} is that $H$ is defined on the whole space $ M$, whereas $V$ is only defined on $M_+^I$ and typically $V(x)\to\infty$ as $x$ approaches $M_0^I$. 
If, in fact, one could extend $V$ to all of $ M$ while keeping condition \textit{\textbf{(i)}} of Definition \ref{def:H-persistence} valid on $ M$, then $$\Lambda^{-}(H)=\Lambda^{+}(H)=0.$$
\end{remark}

A major consequence of $H-$persistence is that it yields the stochastic persistence and in particular the almost sure convergence of $(\Pi_t^x)_{t\geq 0}$ whose limit points lies in $\mathcal P_{inv}(M_0^I)$.

\begin{theorem}[\cite{Benaim2018_Persistence}), Theorem 4.4]\label{thm:HtoStoch}
 Assume that the family $\{(X^{x}_{t})_{t\geq 0}:x\in M_+^I\}$ is $H$‑persistent.
 \begin{enumerate}[label=(\textbf{\roman*})]
  \item For every $x\in M_+^I$, every weak limit point of $(\Pi^{x}_{t})_{t\geq 0}$ lies in $\mathcal{P}_{\mathrm{inv}}(M_+^I)$ a.s.;
  \item The process is stochastically persistent in the sense of Definition \ref{def:stochastic-persistence}.
 \end{enumerate}
\end{theorem}

This theorem has the following immediate consequence.

\begin{corollary}\label{cor:persistent-measure}
 Assume that the family $\{(X^{x}_{t})_{t\geq 0}:x\in M_+^I\}$ is $H$‑persistent and that $\mathcal{P}_{\mathrm{inv}}(M_+^I)$ contains at most one probability measure. Then $\mathcal{P}_{\mathrm{inv}}(M_+^I)$ consists of a single measure, denoted $\{\Pi\}$, and for every $x\in M_+^I$, $$\Pi^{x}_{t} \Rightarrow \Pi,\quad\text{almost surely.}$$
\end{corollary}

When $I=\{1,\cdots, n\}$, we call $\Pi$ the \emph{persistent measure}. In ecological models, $\Pi$ characterizes the long‑term behavior of the coexisting species.

In practice, the study of stochastic persistence of Kolmogorov SDEs follows the ideas from \cite{benaim-24-conv-rate} and \cite{Benaim2018_Persistence}, inspired by \cite{Schreiber2012_Persistence}, \cite{Schreiber2011_Persistence} and \cite{Benaim2008_Robust}. We start by adapting the notion of \emph{invasion rate}: let
\begin{equation}\label{eq:invasion-point}
 \lambda_{i}(x)=F_{i}(x)-\frac{a_{ii}(x)}{2},\quad i=1,\cdots, n,
\end{equation}
be the invasion of species $i$ with respect to $x$. For any $\mu\in\mathcal{P}_{\mathrm{erg}}(M_0^I)$, we write
\begin{equation}\label{eq:invasion-measure}
 \mu\lambda_{i}:=\int_{ M}\lambda_{i}(x)\,\mu(\mathrm dx),
\end{equation}
whenever $\lambda_{i}\in L^{1}(\mu)$. We call $\mu\lambda_{i}$ the \emph{mean invasion rate} of species $i$ with respect to $\mu$. The next theorem extends the Hofbauer criterion (see e.g. \cite{Hofbauer_Criterion}) to possibly degenerate SDEs.

\begin{theorem}[\textbf{Criterion for $H$‐persistence}]\label{thm:invasion-criterion}
 Let $U$ be as in Hypothesis \ref{hyp:main-hyp}. Assume $U^{\frac{1}{2}}$ satisfies the strong law and
 \begin{align}
   \limsup_{||x||\to\infty}\frac{U^{\frac{1 }{2}}(x)}{1+\sum_{i\in I}|F_{i}(x)|}=\infty.\label{eq:growth} 
 \end{align}
 \begin{enumerate}[label=(\textbf{\roman*})]
  \item For every $\mu\in\mathcal{P}_{\mathrm{erg}}(M_0^I)$ and $i\in I$, one has $\lambda_{i}\in L^{1}(\mu)$ and $$\mu\lambda_{i}\neq 0\quad\Longrightarrow\quad \operatorname{supp}(\mu)\subset M_0^{(i)}.$$
  \item If in addition there exist positive numbers $\{p_{i}\}_{i\in I}$ such that
  \begin{equation}\label{eq:Hofbauer-cond}
   \sum_{i\in I} p_{i}\,\mu\lambda_{i}>0,\quad\text{for all }\mu\in\mathcal{P}_{\mathrm{erg}}(M_0^I),
  \end{equation}
  then $\{(X^{x}_{t})_{t\geq 0}:x\in M_+^I\}$ is $H$–persistent.
 \end{enumerate}
\end{theorem}

The proof is detailed in the Appendix \ref{appendix-proof-invasion-rate}. It is inspired by the proof of Theorem 5.1 in \cite{Benaim2018_Persistence}. In particular, the introduction of the extended generators and carré du champ makes the construction of $V$ and $H$ easier, in order to prove that the process is $H-$persistent.

\subsection[Convergence a.s. and in T.V.]{Convergence almost sure and in Total variation}\label{section:convergence}

We recall some core definitions that will lead to the convergence almost sure of $(\Pi_t^x)_{t\geq 0}$ and in Total variation of $(P_t(x,\cdot))_{t\geq 0}$. 

\begin{definition}\label{def:accessible-point}
 A point $y\in M$ is \emph{accessible} from $x\in M$ if, for every neighborhood $U$ of $y$, there exists $t\geq 0$ with $P_{t}(x,U)>0$. We denote by $\Gamma_{x}$ the set of points accessible from $x$. For $A\subset M$ set $\displaystyle \Gamma_{A}=\cap_{x\in A}\Gamma_{x}$.
\end{definition} 

Let $G$ be the 1–resolvent kernel associated to $(P_t)_{t\geq 0}$ and defined in (\ref{eq:1-dim-resolvent-kernel}). Accessibility is also recoverable through $G$ (see e.g Section 5.2.1 and Proposition 5.19 in \cite{Benaim2022_Markov}) as $$\Gamma_x = \text{supp}(G(x,\cdot)).$$

Now, we can state the core notions of \emph{(weak) Doeblin point}.
\begin{definition}\label{def:doeblin}
 A point $x^{\ast}\in M $ is a \emph{weak Doeblin point} if there exist a neighborhood $U$ of $x^{\ast}$ and a non‐zero measure $\xi$ on $ M$ such that, for all $x\in U$, the 1–resolvent kernel $G$ satisfies $$G(x,\cdot)\geq \xi(\cdot).$$
\end{definition}

Equivalently, $U$ is a petite set in the sense of Meyn–Tweedie (see e.g. \cite{Meyn2009_Markov}, Section 5.5).

\begin{definition}\label{def:Doeblin}
 A point $x^{\ast}\in M$ is a \emph{Doeblin point} if there exist a neighborhood $U$ of $x^{\ast}$, a non‐zero measure $\xi$ on $ M$ and a time $t_{\ast}>0$ such that
\begin{equation}\label{eq:doeblin}
  P_{t_{\ast}}(x,\cdot)\geq \xi(\cdot),\quad x\in U.
\end{equation}
\end{definition}

 Equivalently, $U$ is a small set in the sense of Meyn–Tweedie (see e.g. \cite{Meyn2009_Markov}, Section 5.2).

\begin{remark}
 If $x^*$ is an accessible Doeblin point, the minorization condition (\ref{eq:doeblin}) extends to every compact space (see e.g. Lemma 4.9 in \cite{Benaim2018_Persistence}).
\end{remark}

Using the Stratonovich formalism, (\ref{eq:kolmogorov}) writes as
\begin{equation}\label{eq:kolmogorov-stratonovich}
 \mathrm dx_t = S^0(x_t)\,\mathrm dt + \sum_{j=1}^{m} S^j(x_t)\circ dB^j_t
\end{equation}
where, for all $j = 1,\dots,m$ and $i = 1,\dots,n$,
$$ S_{i}^j(x) = x_i \,\Sigma_i^{\,j}(x), \quad\text{and}\quad S^0_{i}(x) = x_i F_i(x) - \frac12 \sum_{j=1}^{m}\sum_{k=1}^{n} \frac{\partial S_i^{j}(x)}{\partial x_k}\,S^{j}_k(x).$$

Let's associate to (\ref{eq:kolmogorov-stratonovich}) the following \emph{deterministic control system},
\begin{equation}\label{eq:deterministic-control-system}
 \dot y(t) = S^0\bigl(y(t)\bigr) + \sum_{j=1}^{m} u_j(t)\,S^j\bigl(y(t)\bigr)
\end{equation}
where the control function $u=(u_1,\dots,u_m):\mathbb{R}_+\to\mathbb{R}^m$ can be chosen to be piecewise continuous. Given such a control function, we let $y(u,x,\cdot)$ denote the maximal solution to \eqref{eq:deterministic-control-system} starting at $x$ in the sense that $y(u,x,0)=x$ (without any assumption about global integrability of vector fields).\\

The following proposition easily follows from the well-known Stroock and Varadhan support theorem in \cite{Stroock1972_Support} (see also Theorem 8.1, Chapter VI in \cite{Ikeda1981_Stochastic}).

\begin{proposition}\label{prop:accessibility-through-control-system}
 Let $x\in M$. A point $p\in M$ lies in $\Gamma_x$ if and only if for every neighborhood $O$ of $p$ there exists a control $u$ such that $y(u,x,\cdot)$ meets $O$ (i.e.\ $y(u,x,t)\in O$ for some $t\geq 0$).
\end{proposition}

Recall that the \emph{Lie bracket} of two smooth vector fields $Y,Z:\mathbb{R}^{n}\to\mathbb{R}^{n}$ is defined by $$[Y,Z](x)=DZ(x)Y(x)-DY(x)Z(x).$$

Given a family $\mathcal{X}$ of smooth vector fields on $\mathbb{R}^{n}$, let $[\mathcal{X}]_{k}$, $k\in\mathbb{N}$, and $[\mathcal{X}]$ be defined by $$[\mathcal{X}]_{0}=\mathcal{X},\quad [\mathcal{X}]_{k+1}=[\mathcal{X}]_{k}\,\cup\bigl\{[Y,Z]:Y,Z\in[\mathcal{X}]_{k}\bigr\},\quad [\mathcal{X}]=\cup_{k}[\mathcal{X}]_{k},$$
and set $[\mathcal{X}](x)=\{Y(x):Y\in[\mathcal{X}]\}$.\\

\begin{definition}\label{def:strong-hormander-condition}
 In the context of (\ref{eq:kolmogorov}) (equivalently (\ref{eq:kolmogorov-stratonovich})), we say that $x^{\ast}\in M$ satisfies the \emph{Hörmander condition} (respectively the \emph{strong Hörmander condition}) if $$\bigl[ \{S^{0},\dots,S^{m}\}\bigr](x^{\ast}),$$ (respectively $\{S^{1}(x^{\ast}),\dots,S^{m}(x^{\ast})\}\cup\{[Y,Z](x^{\ast}):Y,Z\in[\{S^{0},\dots,S^{m}\}]\}$ spans $\mathbb{R}^{n}$.
\end{definition}

The (respectively strong) Hörmander condition will lead to the uniqueness of the invariant measure and the almost sure convergence of $(\Pi_t^x)_{t\geq 0}$ (respectively the convergence in Total variation of $(P_t(x,\cdot))_{t\geq 0}$) towards it.

\begin{corollary}[\cite{Benaim2018_Persistence}, Corollary 5.4]\label{cor:convergence-as-tv-hormander}
 We assume that $\{(X_t^x)_{t\geq 0}:x\in M_+\}$ is $H-$persistent. If there exists $x^{\ast}\in\Gamma_{M_+}\cap M_+$ satisfying the Hörmander condition, then:
 \begin{enumerate}[label=(\textbf{\roman*})]
  \item $\mathcal{P}_{\mathrm{inv}}(M_+)=\{\Pi\}$, $\Pi\ll\lambda,$ and $\Pi^{x}_{t}\Rightarrow\Pi$, $\mathbb{P}$–a.s., for all $x\in M_+$. 
  \item For all $f\in L^{1}(\Pi)$ such that $\int_0^T f(X_s^x)ds<\infty$ for all $T>0$, $\lim_{t\to\infty}\Pi^{x}_{t}f=\Pi f$, $\mathbb{P}$–a.s., for all $x\in M_+$. 
  \item If, in addition, $x^{\ast}$ satisfies the \textit{strong} Hörmander condition, then $(P_t(x,\cdot))_{t\geq 0}$ converges to $\Pi$ in Total variation, for all $x\in M_+$.
 \end{enumerate}
\end{corollary}

As usual, we call the diffusion (\ref{eq:kolmogorov-stratonovich}) \emph{elliptic} at $x$ if $S^1(x),\cdots, S^m(x)$ span $\mathbb R^n$.

\begin{remark}\label{rem:elliptic-case}
Assume that the diffusion (\ref{eq:kolmogorov-stratonovich}) is elliptic at every $x\in M_+$, then the strong Hörmander condition holds on $M_+$. Moreover, by the classical Chow's theorem, every $x\in M_+$ is accessible from $M_+$ (see e.g. \cite{Benaim2022_Markov}, Proposition 6.33).
\end{remark}

\section[Rosenzweig-MacArthur model]{Stochastic persistence and extinction in practice}\label{section:in-practice}

This section details how to study persistence and extinction of two models: we start by the one-dimensional logistic SDE as a basis model that will be used later to study the persistence and extinction of the Rosenzweig-MacArthur model (\ref{eq:rosenzweig-MacArthur-model-in-details}).

\subsection[One-dimensional logistic SDE]{One-dimensional logistic model}

Consider the SDE defined on $M=\mathbb R_+$ as 
\begin{equation}\label{eq:1-dim-logistic-sde}
    \mathrm dz_t = z_t\left[\left(1-\frac{z_t}{\kappa}\right)\mathrm dt + \varepsilon \mathrm dB_t\right], \quad z_0\in\mathbb R_+,
\end{equation}

where $\kappa$, $\varepsilon>0$ and $(B_t)_{t\geq 0}$ is a standard Brownian motion on $\mathbb R$. This model is the \emph{one-dimensional logistic SDE}. Let $M_0=\{0\}$ and $M_+=\mathbb R_+^*$.

Remark that $\forall z\in M$ and $t\geq 0$, the solution to (\ref{eq:1-dim-logistic-sde}) is known explicitly (see e.g. Example 1 in \cite{Benaim2018_Persistence}) as 
\begin{equation}\label{eq:solution-one-dimensional-logistic-sde}
    z_t^z=\frac{ze^{\left(1-\frac{\varepsilon^2}{2}\right)t+\varepsilon B_t}}{1+\frac{z}{\kappa}\int_0^t e^{\left(1-\frac{\varepsilon^2}{2}\right)s+\varepsilon B_s}ds}.
\end{equation}

\begin{proposition}\label{prop:1-dim-logistic-results}
\begin{enumerate}[label=(\textbf{\roman*})]
    \item If $\varepsilon^2>2$, for all $z\in M$, the process $(z_t^z)_{t\geq 0}$ converges almost surely to $0$ and exponentially fast, in particular $$\limsup_{t\to\infty} \frac{1}{t}\log(z_t^z)\leq 1-\frac{\varepsilon^2}{2}<0.$$
    \item If $0<\varepsilon^2<2$, $\{(z_t^z)_{t\geq 0}:z\in M_+\}$ is $H-$persistent, elliptic on $M_+$, and all conclusions of the Corollary \ref{cor:convergence-as-tv-hormander} apply with $\Pi=\gamma_{\varepsilon, \kappa}$ defined as $$\gamma_{\varepsilon, \kappa}(z)\mathrm dz:=\frac{z^{k - 1} \exp\left(-\frac{z}{\theta}\right)}{\Gamma(k)  \theta^k} \mathrm dz.$$
\end{enumerate}
\end{proposition}

\begin{proof}
 \textit{\textbf{(i)}} Since $1-\frac{\varepsilon^2}{2}<0$, remark that $$ze^{\left(1-\frac{\varepsilon^2}{2}\right)t+\varepsilon B_t}=z e^{t\left(1-\frac{\varepsilon^2}{2}+\frac{\varepsilon}{t}B_t\right)}\underset{t\to\infty}{\longrightarrow} 0, \quad \text{almost surely},$$ since $\lim_{t\to\infty}\frac{B_t}{t}=0$. In view of (\ref{eq:solution-one-dimensional-logistic-sde}), $z_t^z\underset{t\to\infty}{\longrightarrow}0$ and 
\begin{equation}\label{eq:lim-sup-bound-logistic-sde}
    \limsup_{t\to\infty} \frac{1}{t}\log(z_t^z)\leq \limsup_{t\to\infty}  1-\frac{\varepsilon^2}{2}+\varepsilon\frac{B_t}{t} =1-\frac{\varepsilon^2}{2}.
\end{equation}

\textit{\textbf{(ii)}} Now, let $0<\varepsilon^2<2$ and $U_{\theta}(z)=e^{\theta z}$ where $\theta < \theta^{*}:=\frac{2}{\varepsilon^2\kappa}$. Then, $U>1$ is a proper map and $$LU_{\theta}(z)=\theta z e^{\theta z}\left(1-z\left(\frac{1}{\kappa}-\frac{\varepsilon^2}{2}\theta\right)\right),$$ where $\frac{1}{\kappa}-\frac{\varepsilon^2}{2}\theta>0$ since $\theta <\theta^*$. Then, $\forall z> \frac{4\kappa}{2-\varepsilon^2\theta\kappa}$, $LU(z)<-\frac{4\theta\kappa}{2-\varepsilon^2\theta\kappa} U_{\theta}(z)$ while $LU(z)<\mathrm{cst}$ if $0\leq z\leq \frac{2\kappa}{2-\varepsilon^2\theta\kappa}$ so Hypothesis \ref{hyp:main-hyp} is verified for all $\theta < \theta^*$.

Since $\Gamma U_{\theta}^{\frac{1}{2}}(z)=\frac{\varepsilon^2\theta^2}{4}z^2e^{\theta z}$, we need a stronger control on $\theta$ to ensure it satisfies the strong law by Corollary \ref{cor:hyp-3-implies}. For any $\theta <\theta^*$, remark that there exists $\delta >0$ small enough such that $\tilde \theta := \theta +\delta<\theta^*$ and $$\Gamma U_{\theta}^{\frac{1}{2}}(z) = \frac{\varepsilon^2 \theta^2}{4}z^2e^{-\delta z}e^{\tilde \theta z}\leq \frac{\varepsilon^2 \theta^2e^{-2}}{\delta ^2}U_{\tilde \theta}(z), \quad \forall z>\frac{2}{\delta}$$ while it is bounded on $z\in\left[0,\frac{2}{\delta}\right]$ so $U_{ \theta}^{\frac{1}{2}}$ satisfies the strong law.

Let $V$ be a smooth function such that $V(z)=-\log(z)$ for $z\in\left]0,\frac{1}{2}\right[$ and $V(z)=0$ for $z\geq1$. Then, $V\in\mathcal{D}_e^{2,+}$ and $$H|_{M_+}(z):= LV(z)=-\left(1-\frac{z}{\kappa}\right)+\frac{\varepsilon^2}{2}, \quad z\in\left]0,\frac{1}{2}\right[,$$ which extends continuously to $M_0$ as $H(0)=-1+\frac{\varepsilon^2}{2}$. 

In particular, $\Gamma (V)(z)$ is bounded so that $V$ satisfies the strong law by Corollary \ref{cor:hyp-3-implies} and $\frac{\sqrt{U}}{1+|H|}$ is proper since $H$ is bounded. The only ergodic probability measure on $M_0$ is $\delta_{0},$ so $$\delta_0 H=H(0)=-1+\frac{\varepsilon^2}{2}<0,$$ which follows from the assumption $0<\varepsilon^2<2$ and it proves that $\{(X_t^x)_{t\geq 0}:x\in M_+\}$ is $H-$persistence.

By ellipticity of (\ref{eq:1-dim-logistic-sde}) on $M_+$ and Remark \ref{rem:elliptic-case}, Corollary \ref{cor:convergence-as-tv-hormander} is applicable and the unique invariant probability measure on $M_+$ is known as $$\gamma_{\varepsilon, \kappa}(z)\mathrm dz=\frac{z^{k - 1} \exp\left(-\frac{z}{\theta}\right)}{\Gamma(k)  \theta^k} \mathrm dz,$$ where $k=\frac{2}{\varepsilon^2} - 1$, $\theta=\frac{\varepsilon^2\kappa}{2}$. One can verify that $$\mathcal L^*\gamma_{\varepsilon, \kappa}(z)=-\frac{\partial}{\partial z}\left[z\left(1-\frac{z}{\kappa}\right)\gamma_{\varepsilon, \kappa}(z)\right]+\frac{\varepsilon^2}{2}\frac{\partial^2}{\partial z^2}\left[z^2\gamma_{\varepsilon, \kappa}(z)\right]=0,$$ since $\frac{\partial}{\partial z}\left[z\left(1-\frac{z}{\kappa}\right)\gamma_{\varepsilon, \kappa}(z)\right]=\gamma_{\varepsilon,\kappa}(z)\left[ k-\frac{2z}{\theta}+\frac{z^2}{2\theta^2} \right]$ and $\frac{\partial^2}{\partial z^2}\left[z^2\gamma_{\varepsilon, \kappa}(z)\right]=\gamma_{\varepsilon,\kappa}(z)\left[ \frac{2}{\varepsilon^2} \left(k -\frac{2z}{\theta}+\frac{z^2}{2\theta^2}\right) \right].$
\end{proof}

In the context of (\ref{eq:rosenzweig-MacArthur-model-in-details}), we achieve Hypothesis \ref{hyp:main-hyp} with the same type of exponential Lyapunov control.

\begin{proposition}\label{prop:main-hypotheses-hold-rosenzweig-MacArthur-model}
 Hypothesis \ref{hyp:main-hyp} is verified with $U(x_1,x_2)=e^{\theta(x_1+x_2)},$ $\forall \theta < \theta^*:=\frac{2}{\kappa\varepsilon^2}$
\end{proposition}

\begin{proof}

Remark that 
\begin{align*}LU(x_1,x_2)&=x_1\left(1-\frac{x_1}{\kappa}-\frac{x_2}{1+x_1}\right)\theta e^{\theta (x_1+x_2)}+x_2\left(-\alpha+\frac{x_1}{1+x_1}\right)\theta e^{\theta (x_1+x_2)} +\frac{\varepsilon^2}{2}\theta^2x_1^2 e^{\theta (x_1+x_2)}\\
&\leq \theta x_1 e^{\theta (x_1+x_2)}\left(\left(1-\frac{x_1}{\kappa}\right) +\frac{\varepsilon^2}{2}\theta x_1\right),\end{align*}

so that $LU\leq -aU+b$ holds for any $\theta <\theta^*$ by analogy with the one-dimensional logistic SDE.
\end{proof}

\subsection[Rosenzweig-MacArthur model]{Extinction for degenerate Rosenzweig-MacArthur model}\label{section:extinction-proof}

Remark that $$\mathrm dx_1=x_1\left[\left(1-\frac{x_1(t)}{\kappa}-\frac{x_2(t)}{1+x_1(t)}\right)\mathrm dt + \varepsilon \mathrm d B_t\right] \leq x_1\left[\left(1-\frac{x_1(t)}{\kappa}\right)\mathrm dt + \varepsilon \mathrm d B_t\right],$$ so that by a comparison theorem (see e.g. \cite{ojm/1200770674}, Theorem 1.1), $x_1(t)\leq z_t$ where $(z_t)_{t\geq 0}$ is solution to the one-dimensional logistic SDE (\ref{eq:1-dim-logistic-sde}) starting at $z_0=x_1(0)$. On the other hand, 
\begin{equation}\label{eq:x-2-solution}
    x_2(t)=x_2(0)\exp\left(-\alpha t + \int_0^t \frac{x_1(s)}{1+x_1(s)}ds\right).
\end{equation}

We show the general extinction of both species in (\ref{eq:rosenzweig-MacArthur-model-in-details}) under the assumption $\varepsilon^2>2$.

\begin{proof}[\textbf{Proof of Theorem \ref{thm:rosenzweig-MacArthur-model-extinction-2}}]

Remark that $z_t\underset{t\to\infty}{\longrightarrow}0$ since $\varepsilon^2>2$, so that $x_1(t)\underset{t\to\infty}{\longrightarrow}0$ almost surely by the comparison remark above and $$\limsup_{t\to\infty}\frac{1}{t}\log(x_2(t))\leq -\alpha + \limsup_{t\to\infty}\frac{1}{t}\int_{0}^t \frac{x_1(s)}{1+x_1(s)}ds=-\alpha.$$

It concludes to prove $(X_t^x)\underset{t\to\infty}{\longrightarrow} (0,0)$ $\mathbb P-$almost surely, $\forall x\in  M_+$. The $\limsup$ bounds are direct using above inequality and the logistic one (\ref{eq:lim-sup-bound-logistic-sde}) for $x_1$.
\end{proof}

We give the proof of the extinction of $x_2$ under the assumptions $0<\varepsilon^2<2$ and $\Lambda(\varepsilon, \alpha,\kappa)<0$.

\begin{proof}[\textbf{Proof of Theorem \ref{thm:rosenzweig-MacArthur-model-extinction-1}}]

From the one-dimensional logistic SDE (\ref{eq:1-dim-logistic-sde}), we know that $\gamma_{\varepsilon, \kappa}$ is the unique invariant probability measure supported on $]0,+\infty[$ and $\Pi_t^z\Rightarrow \gamma_{\varepsilon, \kappa}$ for any $z>0$ for (\ref{eq:1-dim-logistic-sde}). It follows that $$\frac{1}{t}\int_0^t \frac{z(s)}{1+z(s)}ds \underset{t\to\infty}{\longrightarrow} \int_0^{+\infty} \frac{z}{1+z}\gamma_{\varepsilon,\kappa}(z)\mathrm dz, \quad \text{almost surely.}$$
    
    Since $x_1(t)\leq z_t$ for all $t\geq 0$ almost surely, where $z_t$ is the solution of (\ref{eq:1-dim-logistic-sde}) starting at $z_0=x_1(0)$ and since $u\mapsto\frac{u}{1+u}$ is an increasing function on $\mathbb R_+$, then for all $t>0$, $$\frac{1}{t}\int_0^t \frac{x_1(s)}{1+x_1(s)}ds\leq \frac{1}{t}\int_0^t \frac{z(s)}{1+z(s)}ds.$$ Hence, $$\limsup_{t\to\infty}-\alpha + \frac{1}{t}\int_0^t \frac{x_1(s)}{1+x_1(s)}ds\leq -\alpha + \int_0^{+\infty} \frac{x}{1+x}\gamma_{\varepsilon,\kappa}(x)\mathrm dx =\Lambda(\varepsilon, \alpha, \kappa), \quad \text{almost surely,}$$ and from (\ref{eq:x-2-solution}), we obtain $$\limsup_{t\to\infty} \frac{1}{t}\log(x_2(t))\leq  \Lambda(\varepsilon, \alpha, \kappa), \quad \text{almost surely.}$$
    
    This concludes the proof of the extinction of $x_2$ whenever $\Lambda(\varepsilon, \alpha, \kappa)<0$.
\end{proof}

\begin{proof}[\textbf{Proof of Proposition \ref{prop:convergence-when-x-2-extinction}}]
    We will prove that $\{(X_t^x)_{t\geq 0}:x\in M_+^{(1)}\}$ is $H-$persistent. Indeed, Hypothesis \ref{hyp:main-hyp} is verified by Proposition \ref{prop:main-hypotheses-hold-rosenzweig-MacArthur-model} and let $V$ be a smooth function such that $V(x_1,x_2)=-\log(x_1)$ for $x_1\in\left]0,\frac{1}{2}\right[$ and $V(x_1,x_2)=0$ if $x_1\geq 1$ so that $V$ is positive, $\lim_{x_1\to 0}V(x)=+\infty$ and $$LV(x_1,x_2)|_{M_+^{(1)}} = -\left(1-\frac{x_1}{\kappa}-\frac{x_2}{1+x_1}\right) +\frac{\varepsilon^2}{2}=:H(x_1,x_2)|_{M_+^{(1)}}, \quad \forall 0<x_1<\frac{1}{2},$$ 
    
    which extends continuously to $-1+x_2+\frac{\varepsilon^2}{2}$ at $x_1=0$. Since $\Gamma V(x_1,x_2)$ is bounded, $V$ satisfies the strong law by Corollary \ref{cor:hyp-3-implies}. 

    Note that $\Gamma U_{\theta}^{\frac{1}{2}}(x_1,x_2)=\frac{\varepsilon^2\theta^2}{4}x_1^2e^{\theta(x_1+x_2)}$ so by using the same trick than in the proof of Proposition \ref{prop:1-dim-logistic-results}\textit{\textbf{(ii)}}, $\Gamma U_{\theta}^{\frac{1}{2}}\leq \frac{\varepsilon^2\theta e^{-2}}{\delta^2}U_{\tilde \theta}(x_1,x_2)+\mathrm{cst}$ where $\tilde \theta: = \theta +\delta <\theta^*$ for $\delta >0$ small enough and $U_{\theta}^{\frac{1}{2}}$ satisfies the strong law.
    
    Moreover, $\frac{\sqrt{U}}{1+|H|}$ is proper since $H$ is at most polynomial while $U$ is exponential. Finally, since the only invariant probability measure on $M_0^{(1)}$ is $\delta_0(x_1)\otimes\delta_0(x_2)$, we have $$\mu H=H(0,0)=-1+\frac{\varepsilon^2}{2}<0,$$ since $0<\varepsilon^2<2$. By $H-$persistence and Corollary \ref{cor:persistent-measure}, it implies that $\Pi(\mathrm dx_1,\mathrm dx_2)=\gamma_{\varepsilon,\kappa}(x_1)\mathrm dx_1\otimes\delta_{0}(x_2)\mathrm dx_2$ is the unique invariant probability measure supported on $M_+^{(1)}$ and $\forall x\in M_+^{(1)}$, $\Pi_t^x\Rightarrow \Pi$.
\end{proof}

\begin{remark}
    Given Proposition \ref{prop:convergence-when-x-2-extinction}, the bound from Theorem \ref{thm:rosenzweig-MacArthur-model-extinction-1} is exact since $$\limsup_{t\to\infty}\frac{1}{t}\log(x_2(t))=-\alpha + \Pi\left(\frac{x_1}{1+x_1}\right)=\Lambda(\varepsilon,\alpha,\kappa).$$
\end{remark}

\subsection[Rosenzweig-MacArthur model]{Stochastic persistence for degenerate Rosenzweig-MacArthur model}

First of all, we show that the main assumptions to apply the stochastic persistence methodology hold true. Here, we use a different form of $U$ than the one from \cite{Benaim2018_Persistence} and the proofs include the tools newly integrated throughout this text. 

We focus on the SDE described in (\ref{eq:rosenzweig-MacArthur-model-in-details}) and we show that conditions (\ref{eq:growth}) and (\ref{eq:Hofbauer-cond}) from Theorem \ref{thm:invasion-criterion} hold, and in particular the process satisfies the $H-$persistence condition.

\begin{proposition}\label{prop:stoch-persistence-rosenzweig-MacArthur-model}
 Condition (\ref{eq:growth}) is verified, which is $$\limsup_{||x||\to\infty}\frac{U^{\frac{1 }{2}}(x)}{1+|F_{1}(x)|+|F_{2}(x)|}=\infty.$$

 Furthermore, if we suppose that $0<\varepsilon^2<2$ and $\Lambda(\varepsilon,\alpha, \kappa)>0$, there exists positive numbers $p_1$, $p_2$ such that $$p_{1}\,\mu\lambda_{1}+p_{2}\,\mu\lambda_{2}>0,\quad\text{for all }\mu\in\mathcal{P}_{\mathrm{erg}}(\partial  M),$$ and the process is $H-$persistent.
\end{proposition}

\begin{proof}
Since $U$ is of exponential growth while $F_i$ are polynomial, (\ref{eq:growth}) is obviously verified.

It remains to verify the condition (\ref{eq:Hofbauer-cond}) of Theorem \ref{thm:invasion-criterion}. We aim to evaluate $\mathcal P_{\text{erg}}(M_0)$: since $\mathcal P_{\text{erg}}(M_0)=\mathcal P_{\text{erg}}(M_0^{(1)})\cup \mathcal P_{\text{erg}}(M_0^{(2)})$, it follows from the study of the one-dimensional logistic SDE (see Section 5.1) that $$\mathcal P_{\text{erg}}(M_0^{(2)})=\{\gamma_{\varepsilon,\kappa}\otimes \delta_0\},$$ while $\mathcal P_{\text{erg}}(M_0^{(1)})=\{\delta_{(0,0)}\}$ by the deterministic behavior of $x_2$. We treat each case independently:

\begin{enumerate}
 \item Let $\mu = \delta_{(0, 0)}\in \mathcal P_{\text{erg}}(M_0)$, then:
 \begin{align*}
  p_1\, \delta_{(0, 0)}(\lambda_1) + p_2\, \delta_{(0, 0)}(\lambda_2) = p_1 \cdot \lambda_1(0, 0) + p_2 \cdot \lambda_2(0, 0).
 \end{align*}

 By definition of $\lambda_i$ in (\ref{eq:invasion-point}), $$\lambda_1(0, 0) = F_1(0, 0) - \frac{\varepsilon^2}{2} = 1 - \frac{\varepsilon^2}{2} > 0, \quad \text{since } \varepsilon^2 < 2,$$ and $$ \lambda_2(0, 0) = F_2(0, 0) = -\alpha < 0,$$ thus $p_1\, \mu(\lambda_1)+p_2\, \mu(\lambda_2) > 0$ for $p_1$ large enough, $p_2$ small enough, i.e. $p_1 > \frac{p_2 \alpha}{1 - \frac{\varepsilon^2}{2}}$ for any $p_2>0.$

 \item Let $\mu = \gamma_{\varepsilon,\kappa}\otimes \delta_0 \in \mathcal P_{\text{erg}}(M_0)$. In particular, $\text{supp}(\mu) \nsubseteq M_0^{(1)}$, and by contraposition of Theorem \ref{thm:invasion-criterion}\textit{\textbf{(i)}}, it follows that $\mu(\lambda_1) = 0$ so that condition (\ref{eq:Hofbauer-cond}) becomes $\mu(\lambda_2) > 0$.

 Since $\lambda_2(x_1, x_2) = F_2(x_1, x_2) = -\alpha + \frac{x_1}{1 + x_1}$, $$\mu(\lambda_2) =\int_0^{+\infty} \left( \frac{x}{1 + x} - \alpha \right) \gamma_{\varepsilon, \kappa}(x) \mathrm dx = \Lambda(\varepsilon, \alpha, \kappa),$$ which is strictly positive by assumption. 
\end{enumerate}

The same trick than in the proofs of Proposition \ref{prop:1-dim-logistic-results}\textit{\textbf{(ii)}} and Proposition \ref{prop:convergence-when-x-2-extinction} implies that $U_{\theta}^{\frac{1}{2}}$ satisfies the strong law. Therefore, the conditions of Theorem \ref{thm:invasion-criterion}\textit{\textbf{(ii)}} are satisfied and the process is $H$-persistent.
\end{proof}

\subsection[Rosenzweig-MacArthur model]{Convergence for degenerate Rosenzweig-MacArthur model}

We now aim to prove the existence of points $x^*$ satisfying Hörmander's condition (respectively the strong Hörmander condition) in $\Gamma_{ M_+}$, in order to show almost sure convergence (respectively in Total variation) of the occupation measure (respectively of $(P_t(x,\cdot))_{t\geq 0}$) towards the unique invariant probability. We focus on our model of interest (\ref{eq:rosenzweig-MacArthur-model-in-details}).

\begin{proposition}\label{prop:strong-hormander-condition-rosenzweig-MacArthur-model}
 The strong Hörmander condition holds for every points $x^*\in  M_+$.
\end{proposition}

\begin{proof}
Rewriting (\ref{eq:rosenzweig-MacArthur-model-in-details}) using the Stratonovich formalism introduced in (\ref{eq:kolmogorov-stratonovich}), it yields $$S_1^1(x) = x_1 \cdot \varepsilon, \ S_2^1(x) = S_1^2(x) = S_2^2(x) = 0,$$

and therefore
$$S^0(x) = \begin{pmatrix} x_1 (F_1(x_1, x_2) - \varepsilon^2/2) \\ x_2 F_2(x_1, x_2) \end{pmatrix}, \quad S^1(x) = \begin{pmatrix} x_1 \varepsilon \\ 0 \end{pmatrix}.$$

We obtain
\begin{align*}
 [S^0, S^1](x) &= DS^1(x) \cdot S^0(x) - DS^0(x) \cdot S^1(x) =
 \begin{pmatrix}
  -x_1^2 \varepsilon \cdot \frac{\partial F_1}{\partial x_1}(x) \\
  - x_1 \varepsilon \cdot x_2 \cdot \frac{\partial F_2}{\partial x_1}(x)
 \end{pmatrix}.
\end{align*}

In particular,
\begin{align*}
 \det\left( [S^0, S^1](x),\, S^1(x) \right) &= \det
 \begin{pmatrix}
 -x_1^2 \varepsilon \cdot \frac{\partial F_1}{\partial x_1}(x) & x_1 \varepsilon \\
 -x_1 \varepsilon \cdot x_2 \cdot \frac{\partial F_2}{\partial x_1}(x) & 0
 \end{pmatrix}
 = x_1^2 \varepsilon^2 x_2 \cdot \frac{1}{(1 + x_1)^2} > 0, \quad \forall x \in  M_+,
\end{align*}

which proves that, for all $x^* \in  M_+$, the set $\{[S^0, S^1](x^*),\, S^1(x^*)\}$ spans $\mathbb{R}^2$, and the (strong) Hörmander condition holds on $ M_+$.
\end{proof}

The last step is to show that $\Gamma_{ M_+}$ contains $ M_+$. Here, we use the notion of control system introduced in (\ref{eq:deterministic-control-system}). 

\begin{proposition}\label{prop:accessibility-rosenzweig-MacArthur-model}
 $\Gamma_{ M_+}= M_+$ and every points of $ M_+$ is accessible from anywhere in $ M_+$.
\end{proposition}

\begin{proof}
In our model, we introduce a new control variable $$v=\varepsilon u- \frac{\varepsilon^2}{2} \ \Leftrightarrow \ u=\frac{v}{\varepsilon} +\frac{\varepsilon}{2},$$ so that the associated control system $y(t)$ is solution of the differential equation
\begin{align}\label{eq:control-system-r-m-model}
 \begin{split}
  \dot y(t) &= S^0\left(y(t)\right) +\sum_{j=1}^m u^j S^j\left(y(t)\right)\\
  &= 
  \begin{pmatrix}
  y_1(t)\left(F_1(y_1(t),y_2(t))-\varepsilon^2/2\right)\\ y_2(t) F_2(y_1(t),y_2(t)) \end{pmatrix}
  +
  \begin{pmatrix}\varepsilon y_1(t)\\ 0\end{pmatrix}
  \cdot
  \begin{pmatrix} u\\ 0 \end{pmatrix}\\ 
  &= 
  \begin{pmatrix} y_1(t)\left(F_1(y_1(t),y_2(t))+ v \right) \\ y_2(t)F_2(y_1(t),y_2(t))\end{pmatrix},
 \end{split}
\end{align}

Let $L$ be the vertical line defined by $x_1=\frac{\alpha}{1-\alpha}$, and let $x_2\geq 0$: it implies that, along this line, $\mathrm dx_2=0$ since 
\begin{align*}
 F_2(x_1, x_2)&=F_2\left(\frac{\alpha}{1-\alpha}, x_2\right)= -\alpha + \alpha=0.
\end{align*}

Before this line, the drift associated to the second coordinate of the control system $y(t)$ is negative, i.e. $y_2(t)$ is decreasing, while it is increasing after the line.\\

Let $P_v(x_1,x_2)$ be the parabola defined by
\begin{align*}\label{eq:parabola}
 (1+v-\frac{x_1}{\kappa})(1+x_1)=x_2 \ &\Leftrightarrow \ -\frac{1}{\kappa}x_1^2+\left(v-\frac{1}{\kappa}+1\right)x_1 + 1+v = x_2.\\
 & \Leftrightarrow \ F_1(x_1,x_2)+v=0. 
\end{align*}

This last equality implies that below $P_v$, the drift associated to the first coordinate of the control system $y(t)$ is positive, i,e, $y_1(t)$ is increasing, while it is decreasing above $P_v$.

It naturally implies that $P_v$ reaches its maximum $(x_1^*, x_2^*)$ at
\begin{align*}
 x_1^*&=\frac{1}{2}\left(\kappa v - 1 + \kappa \right),\\
 x_2^*&= \left(1+v-\frac{(\kappa v-1+\kappa)}{2\kappa}\right)\left((1+\frac{(\kappa v-1+\kappa)}{2}\right)=\frac{(\kappa (v+1) + 1)^2}{4\kappa}.
\end{align*}

Let $z=(z_1,z_2)\in  M_+$ and $O_z$ be an open neighborhood of $z$: we can choose $v^*$ sufficiently large so that $z$ is always below the graph of $P_{v^*}$ and such that the $x_1-$coordinate of the maximum of $P_{v^*}$ is above $\frac{\alpha}{1-\alpha}$.

For a small-enough open neighborhood of the origin $O$, $\forall x\in O\cap  M_+$, the control system $t\mapsto y(v^*, x,t)$ remains below $P_{v*}$ until it crosses $P_{v^*}$ near $(\kappa(1+v^*), 0)$. Then, it remains above $P_{v^*}$ while being decreasing on $x_1$, increasing on $x_2$ until it crosses $L$, when it becomes decreasing on both $x_1$ and $x_2$. In particular, it crosses $x_2=z_2$ while $x_1>z_1$.\\

Let's construct a piecewise constant control $v(t)$ as follows:
\begin{enumerate}[label=(\textit{\roman*})]
 \item $v(t)=-1$ until $t\mapsto y(v,x,t)$ enters $O$.
 \item Then, $v(t)=v^*$ until $t\mapsto y(v,x,t)$ crosses $x_2=z_2$ (after crossing $P_{v^*}$).
 \item Finally, $v(t)=-R$ for $R$ large enough so that $t\mapsto y(v,x,t)$ enters $O_z$, i.e. decreasing on $x_1$ as fast as possible while remaining in a small-enough neighborhood of $x_2=z_2$.
\end{enumerate}

It follows from Stroock and Varadhan support theorem (see Proposition \ref{prop:accessibility-through-control-system}) that $ M_+=\Gamma_{ M_+}$ since any point $z\in  M_+$ is accessible from any point $x\in  M_+$, i.e. $\Gamma_{ M_+}=\cap_{z\in  M_+} \Gamma_z= M_+$.
\end{proof}

Figure \ref{fig:control-system-r-m-model} exhibits the structure of the associated control system through an example: starting at $x=(0.3,0.3)$, $z=(1,2)$ with $O_z=B_{0.15}(z),$ $O=B_{0.15}(0,0)$, $\alpha=0.3$, $\kappa=0.5$ and $v^*=3$.

Starting at $x$, the control system is sent to $O\cap M$, then go to the right until it passes $P_{v^*}$, moves up until it crosses the line $x_2=z_2$ and is finally sent to $O_z$.

\begin{figure}[H]
 \centering
 \includegraphics[width=0.5\textwidth]{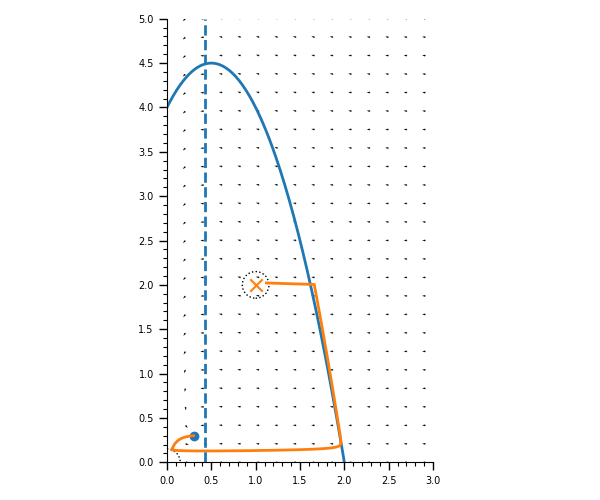}
 \caption{Example of (\ref{eq:control-system-r-m-model}) with $x=(0.3,0.3)$, $z=(1,2)$ with $O_z=B_{0.15}(z),$ $O=B_{0.15}(0,0)$, $\alpha=0.3$, $\kappa=0.5$ and $v^*=3$. The trajectory of $y(x,v,t)$ with the proper $v$ as defined before is depicted in orange, with the blue parabola $P_{v^*}$ whose maximum is attained after the vertical blue dashed line $L$. }
 \label{fig:control-system-r-m-model}
\end{figure}

As a consequence of Propositions \ref{prop:stoch-persistence-rosenzweig-MacArthur-model}, \ref{prop:strong-hormander-condition-rosenzweig-MacArthur-model} and \ref{prop:accessibility-rosenzweig-MacArthur-model}, we now obtain:
\begin{corollary}\label{cor:conclusions-holds-rosenzweig-macarthur-model}
    In (\ref{eq:rosenzweig-MacArthur-model-in-details}), every point $x\in M_+$ is accessible from $M_+$, satisfies the strong Hörmander condition and $\{(X_t^x)_{t\geq 0}:x\in M_+\}$ is $H-$persistent so that all conclusions of the Corollary \ref{cor:convergence-as-tv-hormander} apply.
\end{corollary}

\section{Convergence rate}\label{section:convergence-rate}

The last question regards the rate of convergence in Corollaries \ref{cor:convergence-as-tv-hormander} and \ref{cor:conclusions-holds-rosenzweig-macarthur-model}.

Exhibiting an exponential convergence rate is relatively direct when the extinction set is compact as proposed in Theorem A.12 from \cite{benaim-24-conv-rate}. However, a non-compact extinction set as $\partial R_+^n$ is more complex to handle. In particular, we have to ensure that $H$ is strictly negative outside a compact subset of $ M$. In the case of (\ref{eq:kolmogorov}), this idea is called \emph{H-persistence at infinity}, as detailed in Section 4.5 from \cite{Benaim2018_Persistence}.

Unfortunately, the conditions to ensure such a stronger $H-$persistence property fails in the case of (\ref{eq:rosenzweig-MacArthur-model-in-details}). As proposed in Theorem 5.1\textbf{\textit{(iii)}} and Example 5 in \cite{Benaim2018_Persistence}, typical models such that $H-$persistence at infinity holds occurred when the drifts are bounded below, in the sense that $\liminf_{||x||\to+\infty} F_i(x)>-\infty$, which is not verified in (\ref{eq:rosenzweig-MacArthur-model-in-details}). 

However, we can prove a polynomial convergence rate by following \cite{Benaim_Antoine_2022}. We strengthen Hypothesis \ref{hyp:main-hyp} as follows:
\begin{hypothesis}\label{hyp:main-hyp-strong}
 In addition to Hypothesis \ref{hyp:main-hyp}, we suppose that the function $U: M\to [1,\infty[$ satisfies 
 \begin{equation}\label{eq:gamma-condition-just-for-convergence-rate}
     \Gamma  (U)\leq cU^2,
 \end{equation}

 where $\Gamma$ is defined in (\ref{eq:extended-carre-du-champ-kolmogorov-general}),
 \begin{equation}\label{eq:U-ln-x-positive}
  \liminf_{||x||\to\infty} \frac{U(x)}{\ln(||x||)}>0,
 \end{equation}
 \begin{equation}\label{eq:LU-F-i-negative}
  \limsup_{||x||\to\infty} \left(LU(x)+p_0\sum_{i=1}^n |F_i(x)|\right)<0, \quad \text{for some } p_0>0,
 \end{equation}
 and
 \begin{equation}\label{eq:F-i-bounded-U}
  \sum_{i=1}^n |F_i(x)|\leq CU^{d_0}(x), \quad \text{for some } d_0\geq 1 \text{ and } C>0.
 \end{equation}
\end{hypothesis}

We also assume that condition (\ref{eq:Hofbauer-cond}) from Theorem \ref{thm:invasion-criterion} is verified for $I=(1,\cdots, n)$: we notice that $\sum_{i=1}^n p_i$ can be assumed to be sufficiently small without loss of generality. Then, in view of the definition of $H-$persistence, we modify the construction of $V$ from Theorem \ref{thm:invasion-criterion} so that $$V(x)=U(x)+\sum_{i=1}^n p_i \ln\left(\frac{1}{x_i}\right), \quad \forall x\in M_+.$$

For $\sum_{i=1}^n p_i$ sufficiently small, it implies that $V$ is positive on $M_+$ and from (\ref{eq:U-ln-x-positive}), we deduce $$V(x)\to \infty\quad \text{ as } ||x||\to\infty.$$ Thus, we obtain $$H|_{M_+}=LV(x)=LU(x)+\sum_{i=1}^np_i\lambda_i,$$ which can be extended continuously to $M_0$. In particular, since $U$ is defined on all $ M$, the persistence condition $\Lambda^-(H)>0$ is still verified if $\sum_{i=1}^n p_i\mu(\lambda_i)>0$ for all $\mu\in\mathcal P_{\mathrm{erg}}(M_0)$ since $\Lambda^-(U)=0$ (see Remark \ref{rem:VH-domain}).\\

In these settings, let $q_0>1$ be a constant such that $$-a+\frac{q_0-1}{2}c=0,$$ where $a>0$ is the constant from condition (\ref{eq:hyp-LU}) in Hypothesis \ref{hyp:main-hyp} and $c>0$ the one from (\ref{eq:gamma-condition-just-for-convergence-rate}) in Hypothesis \ref{hyp:main-hyp-strong}. Then, for $q\in\Big]1,\min\{q_0, \frac{q_0+2}{2}\}\Big[$, we define 
\begin{equation}\label{eq:definition-W-q}
 W_q=
 \begin{cases}
  V^q + CU^q, &\text{if } 1<q\leq2,\\
  V^q + CU^{2q-2}, &\text{if } q>2,
 \end{cases}
\end{equation}
where $C$ is the positive constant from (\ref{eq:F-i-bounded-U}).

\begin{theorem}[\cite{Benaim_Antoine_2022}, Theorem 4.1]\label{thm:polynomial-conv-rate}
 Assume that there exists a map $U: M\to[1,\infty[$ satisfying Hypothesis \ref{hyp:main-hyp-strong} and condition (\ref{eq:Hofbauer-cond}) from Theorem \ref{thm:invasion-criterion}. Moreover, we suppose that there exists $x^*\in \Gamma_{M_+}\cap M_+$ which satisfies the strong Hörmander condition. Then, for all $q\in\Big]1,\min\{q_0, \frac{q_0+2}{2}\}\Big[$ and for all $1\leq \beta \leq q$,
 \begin{equation}\label{eq:polynomial-rate-of-convergence}
  \lim_{t\to\infty} t^{\beta-1}||P_t(x,\cdot)-\Pi(\cdot)||_{W_{\beta, q}}=0, \quad \forall x\in M_+, 
 \end{equation}
 where $W_{\beta, q}=W_q^{1-\beta/q}$ with $W_q$ defined as in (\ref{eq:definition-W-q}), $||f||_{W_{\beta,q}}=\sup_{x\in M_+} \frac{|f(x)|}{1+W_{\beta,q}(x)}, \ \forall f\in\mathcal{B}(M_+),$ and $$\|\mu\|_{W_{\beta,q}}=\sup_{|g|\leq W_{\beta,q}}|\mu(g)|,$$ for all signed measures $\mu$.
\end{theorem}

\begin{remark}\label{rem:convergence-rate-implies-total-variation}
    Since $1-\frac{\beta}{q}>0$ and $W_q\geq 1$, (\ref{eq:polynomial-rate-of-convergence}) naturally implies that $$\lim_{t\to\infty}t^{\lambda}||P_t(x,\cdot)-\Pi(\cdot)||_{TV}=0, \quad \forall x\in M_+,$$ where $\lambda = q-1>0$ and in particular $\beta=q$.
\end{remark}

\subsection[Rosenzweig-MacArthur model]{Polynomial convergence rate for degenerate Rosenzweig-MacArthur model}

Condition (\ref{eq:gamma-condition-just-for-convergence-rate}) fails for $U(x)=e^{\theta(x_1+x_2)}$ since $$\Gamma (U(x_1,x_2))=\varepsilon^2\theta^2x_1^2U^2(x_1,x_2).$$

However, another choice of $U$ satisfies both (\ref{eq:hyp-LU}) and (\ref{eq:gamma-condition-just-for-convergence-rate}) such as $U(x)=1+(x_1+x_2)^n$, for any $n>2$ (see \cite{Benaim2018_Persistence}, Section 5). Under the assumption that $0<\varepsilon^2<2$ and $\Lambda(\varepsilon, \alpha,\kappa)>0$, condition (\ref{eq:Hofbauer-cond}) from Theorem \ref{thm:invasion-criterion} also holds true.

\begin{proposition}\label{prop:conv-rate-polynomial-RM-model}
 Hypothesis \ref{hyp:main-hyp-strong} holds and for all $x\in M_+$, $(P_t(x,\cdot))_{t\geq0}$ converges in Total variation towards $\Pi$ at a polynomial rate.
\end{proposition}

\begin{proof}
Remark that $$\Gamma  (U(x_1,x_2))=\varepsilon^2n^2x_1^2(x_1+x_2)^{2n-2}\leq \varepsilon^2n^2 U^2(x_1,x_2).$$

Condition (\ref{eq:U-ln-x-positive}) is direct since $\ln(||x||)=\frac{1}{2}\ln(x_1^2+x_2^2)\ll x_1^n+x_2^n\leq U(x),$ since $n>2$. 

We also recall that $|F_1(x)|+|F_2(x)|\leq(2+\frac{1}{\kappa}+\alpha)(1+x_1+x_2),$ so that (\ref{eq:F-i-bounded-U}) is verified for $d_0=1$ and $C>2+\frac{1}{\kappa}+\alpha$. In particular, since $LU\leq -a U +b$, for $0<\delta <a$, choose $p_0<\frac{a-\delta}{C}$ such that $$LU(x)+p_0(|F_1(x)|+|F_2(x)|)<-\delta U(x)+b \underset{||x||\to\infty}{\longrightarrow} -\infty,$$ since $U$ is proper, which proves that (\ref{eq:LU-F-i-negative}) is verified.

By Theorem \ref{thm:polynomial-conv-rate}, the convergence rate of $(P_t)_{t\geq 0}$ towards $\Pi$ is polynomial.
\end{proof}

\begin{remark}
We can even provide a precise estimate of the speed of convergence.\\

Since Hypothesis \ref{hyp:main-hyp} holds with $a=\frac{\alpha n}{2}$ and $c=\varepsilon^2n^2$ for any $n>2$. It yields $$q_0=1+\frac{\alpha}{\varepsilon^2n}.$$ Consequently, we obtain
\begin{equation*}
 \min\Bigg\{q_0, \ \frac{q_0+2}{2}\Bigg\}=
 \begin{cases}
  q_0=1+\frac{\alpha}{\varepsilon^2n}, &\text{if } \frac{\alpha}{\varepsilon^2}<\frac{n}{2}, \\
  \frac{q_0+2}{2}=\frac{3}{2}+\frac{\alpha}{2\varepsilon^2n}, &\text{if } \frac{\alpha}{\varepsilon^2}>\frac{n}{2}.
 \end{cases}
\end{equation*}
The speed parameter $\beta>0$ from (\ref{eq:polynomial-rate-of-convergence}) follows
\begin{equation*}
 0<\beta-1<
 \begin{cases}
  \frac{\alpha}{\varepsilon^2n}, &\text{if } \frac{\alpha}{\varepsilon^2}<\frac{n}{2}, \\
  \frac{1}{2}+\frac{\alpha}{2\varepsilon^2n}, &\text{if } \frac{\alpha}{\varepsilon^2}>\frac{n}{2}.
 \end{cases}
\end{equation*}

In the first case, for any parameters $\alpha$, $\varepsilon$ in (\ref{eq:rosenzweig-MacArthur-model-in-details}), we can choose $n$ large enough so that $\frac{\alpha}{\varepsilon^2}<\frac{n}{2}$. In the second case, since $n>2$, the condition is verified when $\frac{\alpha}{\varepsilon^2}>1$ or $\alpha > \varepsilon^2$.
\end{remark}

An additional condition to Hypothesis \ref{hyp:main-hyp-strong} to ensure an exponential convergence rate exposed in \cite{Benaim_Antoine_2022} is a stronger version of (\ref{eq:LU-F-i-negative}),
\begin{equation}\label{eq:stronger-LU-F-i-control}
 \limsup_{||x||\to\infty} \left( L\ln(U)+p_0\sum_{i=1}^n |F_i(x)|\right)<0, \quad \text{for some } p_0>0.
\end{equation} Unfortunately, we observe that $$L\ln(U)=\frac{LU}{U}-\frac{1}{U^2}\Gamma (U),$$ and one can verify that the first term $\frac{LU}{U}$ is of order $x_1$ while the second one is constant such that we cannot control $\sum_{i=1}^n |F_i(x)|$, which is of order $x_1+x_2$. In particular, when $x_1\to 0$ and $x_2\to\infty$, we cannot verify (\ref{eq:stronger-LU-F-i-control}) whenever $p_0>0$ is small.

These types of conditions are similar to those stated in \cite{Hening2018_Coexistence} to ensure exponential convergence rate when the extinction set is non-compact as in our situation.

\section{Proof of Theorem \ref{thm:rosenzweig-MacArthur-model-persistence}}\label{section:appendix-mainresults}

Note that Hypothesis \ref{hyp:main-hyp} is already verified by Proposition \ref{prop:main-hypotheses-hold-rosenzweig-MacArthur-model}, while the $H-$persistence condition is also verified by Proposition \ref{prop:stoch-persistence-rosenzweig-MacArthur-model}.\\

\textit{\textbf{(i)}} The strong Hörmander condition holds for (\ref{eq:rosenzweig-MacArthur-model-in-details}) on all $ M_+$ by Proposition \ref{prop:strong-hormander-condition-rosenzweig-MacArthur-model}, and every points of $ M_+$ is accessible from $ M_+$ by Proposition \ref{prop:accessibility-rosenzweig-MacArthur-model}. Following Corollary \ref{cor:convergence-as-tv-hormander}, for every point $x\in  M_+$,$$\Pi^{x}_{t}\Rightarrow\Pi, \quad \mathbb{P}-\text{almost surely }.$$ 

\textit{\textbf{(ii)}} The first convergence result is a direct consequence of Corollary \ref{cor:convergence-as-tv-hormander}\textit{\textbf{(ii)}}, where the additional condition $\int_0^T f(X_s^x)ds$ for all $T>0$ completes the proof of Proposition 4.8 in \cite{Benaim2018_Persistence}. 

In addition, by Proposition \ref{prop:conditions-hypotheses-hold}\textit{\textbf{(i)}}, even if we lack information about the persistence measure $\Pi$, we know that $\Pi$ admits an exponential moment of order $\theta$, hence its tails decay at least exponentially, since $$\int_{M} e^{\theta(x_1+x_2)}\Pi(\mathrm dx_1,\mathrm dx_2)<\infty.$$

\textit{\textbf{(iii)}} Also by Corollary \ref{cor:convergence-as-tv-hormander}, for all $x\in M_+$, $(P_t(x,\cdot))_{t\geq 0}$ converges to $\Pi$ in Total variation.\\

For the convergence rate, we rely on Proposition \ref{prop:conv-rate-polynomial-RM-model} to show that the conditions of Theorem \ref{thm:polynomial-conv-rate} hold. It follows that the convergence rate is polynomial since there exists $x^*\in \Gamma_{ M_+}\cap  M_+$ satisfying the strong Hörmander condition.\\

\textit{\textbf{(iv)}} As stated in Corollary \ref{cor:convergence-as-tv-hormander}, we automatically observe $\Pi\ll \lambda$. \qed

\appendix

\section{Appendix}\label{section:appendix}

\subsection{Proof of Proposition \ref{prop:hyp-LU-implies}}\label{appendix-proof-hyp-LU-implies}

Since the drift and the covariance are supposed to be locally Lipschitz continuous, we can use classical results on stochastic differential equations such that there exists for any $x\in M$ a unique continuous process
$(X_t)_{t\geq 0}$ defined on some interval $[0,\tau^x[$ solution to (\ref{eq:kolmogorov}), with initial condition
$X_0=x$ and such that $t<\tau^x\iff \|X_t\|<\infty$ (see e.g. \cite{SDE-on-Manifold}, Theorem 1.1.8), denoted as $(X_t^x)_{0\leq t<\tau^x}$.

Furthermore, applying Itô’s formula to $Y_{t,i}^x:=\ln(X_{t,i}^x)$, rearranging the terms, and using the uniqueness of the solutions, then
$$X^x_{t,i} = x_i\exp\Bigl(\int_0^t\bigl[F_i(X^x_s)-\tfrac12a_{ii}(X^x_s)\bigr] ds +\sum_j\int_0^t\Sigma^j_i(X_s) dB^j_s\Bigr).$$
Thus
\begin{equation}\label{eq:solution-positive-starting-out-zero}
 x_i>0 \Longrightarrow X^x_{t,i}>0, \ \forall t\in[0,\tau^x[,
\end{equation}
and
\begin{equation}\label{eq:solution-null-starting-zero}
 x_i=0 \Longrightarrow X^x_{t,i}=0\ \forall t\in[0,\tau^x[.
\end{equation}

Now, we prove that $\tau^x=\infty$. For any $C^2$ function $\psi\colon M\to\mathbb{R}$, by Itô’s formula, 
\begin{equation}\label{eq:ito-for-U-function}
 \psi\bigl(X^x_t\bigr)-\psi(x)-\int_0^t L\psi\bigl(X^x_s\bigr) ds =\sum_{i=1}^n\int_0^t\frac{\partial\psi}{\partial x_i}\bigl(X^x_s\bigr) \left[X^x_{s,i}\sum_{j=1}^m\Sigma^j_i\bigl(X^x_s\bigr)\right] dB^j_s.
\end{equation}

Let $\tau_k^x=\inf\{t\ge0:U(X^x_t)\geq k\},\ k\in\mathbb{N}$. It follows Hypothesis \ref{hyp:main-hyp} and (\ref{eq:ito-for-U-function}) that for all $x\in M$,
\begin{align*}
 k \mathbb{P}\left(\tau_k^x\leq t\right)
 &=\mathbb{E}\left[U\left(X^x_{\tau_k^x}\right)\mathbf{1}_{\{\tau_k^x\leq t\}}\right]\\
 &\le\mathbb{E}\left[U\left(X^x_{t\wedge\tau_k^x}\right)\right]\\
 &= U(x)+\mathbb{E}\left[\int_0^{t\wedge\tau_k^x} LU\left(X^x_s\right) ds\right]\\
 &\leq U(x)-a \mathbb{E}\left[\int_0^{t\wedge\tau_k^x}U\left(X^x_s\right) ds\right]+b t\\
 &\leq U(x)+b t.
\end{align*}
Hence $$ \mathbb{P}(\tau^x\leq t) =\mathbb{P}\Bigl(\cap_{k\ge0}\{\tau_k^x\leq t\}\Bigr) =\lim_{k\to\infty}\mathbb{P}(\tau_k^x\leq t) =0, $$ proving that $\tau^x=\infty$ almost surely.

We now let $(P_t)_{t\geq 0}$ denote the semigroup acting on bounded (respectively non‑negative) measurable functions $f\colon M\to\mathbb{R}$, defined by $P_t f(x)=\mathbb{E}\bigl(f(X^x_t)\bigr).$ Then, $C_b( M)$–Feller continuity just follows from the dominated convergence theorem and the continuity in $x$ of the solution $(X^x_t)_{t\geq 0}$, which implies that $x\mapsto P_tf(x)$ is continuous for every bounded continuous $f$.\qed

\begin{remark}\label{rem:invariant-extinction-set}
 From \eqref{eq:solution-positive-starting-out-zero} and \eqref{eq:solution-null-starting-zero}, it follows automatically that $M_+^I$ and $M_0^I$ are invariant under $(P_t)_{t\geq 0}$, for all $I\subset\{1,\cdots,n\}$.
\end{remark}

\subsection{Proof of Proposition \ref{prop:conditions-hypotheses-hold}}

Before stating the proof of Proposition \ref{prop:conditions-hypotheses-hold}, we introduce some useful intermediary results. The following proposition is the continuous version of a well-known one for discrete–time Feller chains, and its proof is directly based on the discrete-time case.

\begin{proposition}\label{prop:almost sure-limit-point-proba-inv}
 For all $x\in M$, almost surely, every limit point of $\bigl(\Pi^{x}_{t}\bigr)_{t\geq 0}$ lies in $\mathcal{P}_{\mathrm{inv}} ( M)$ $\mathbb P$-almost surely.
\end{proposition}

\begin{proof}
Given $G$ be the \emph{1–resolvent} (\ref{eq:1-dim-resolvent-kernel}), let $(Y^{x}_{n})_{n\geq 0}$ be the discrete time chain obtained by sampling the continuous one $(X_{t})_{t\geq 0}$ at random times exponentially distributed, in the sense that
\begin{equation}\label{eq:sampled-discrete-time-chain}
 Y_n^x=X_{T_n}^x,
\end{equation}
where $T_{0}=0, \cdots, T_{n+1}=T_{n}+U_{n+1}$ and $U_1, U_2, \cdots$ is a sequence of independent identically distributed random variables having an exponential distribution with parameter $1$ and independent of $(X^{x}_{t})_{t\geq 0}$. We set
$$\widehat\Pi^{x}_{n}=\frac{1}{n}\sum_{k=0}^{n-1}\delta_{Y^{x}_{k}},$$ 
the empirical measure of the discrete chain $(Y^{x}_{n})_{n\geq 0}$, where $\delta_z$ is the Dirac measure that assigns mass $1$ to $\{z\}$. 

For a discrete–time Feller chain, it is classical that the limit points of the empirical measures are invariant (see e.g. \cite{Benaim2022_Markov}, Theorem 4.20). Since it is invariant with respect to its kernel $G$, limit points of $(\widehat\Pi^{x}_{n})_{n\geq 0}$ lie almost surely in $\mathcal{P}_{\mathrm{inv}}( M)$ (see e.g. \cite{Benaim2022_Markov}, Proposition 4.57).\\

We now show that $(\Pi^{x}_{t})_{t\geq 0}$ and $(\widehat\Pi^{x}_{\lfloor t\rfloor})_{t\geq 0}$ possess the same limit points almost surely, where $\lfloor \cdot \rfloor$ is the floor function. On $ M$, there exists a countable family $\{f_{n}\}_{n\geq 0}\subset C_{b}( M)$ such that
$$D(\mu,\nu):=\sum_{k\geq 0}\frac{1}{2^{k}} \min\bigl(|\mu f_{k}-\nu f_{k}|,1\bigr),$$
is a distance on $\mathcal P( M)$ that metricizes the topology of weak convergence on $\mathcal{P}( M)$ (see e.g. \cite{Benaim2022_Markov}, Proposition 4.5). That is, $$\mu_n \Rightarrow \mu \ \Leftrightarrow \ D(\mu_n,\mu)\underset{n\to\infty}{\longrightarrow} 0.$$ By Proposition 4.58 in \cite{Benaim2022_Markov}, for every $f\in C_{b}( M)$, $$\lim_{t\to\infty} \frac{1}{t} \int_0^tf(X_s)ds-\frac{1}{\lfloor t\rfloor}\sum_{k=0}^{\lfloor t\rfloor -1}Gf(Y_k)=0,$$ almost surely, which is $$\lim_{t\to\infty}\bigl|\Pi^{x}_{t}(f)-\widehat\Pi^{x}_{\lfloor t\rfloor}Gf\bigr|=0\quad\text{a.s.}$$

Moreover, let $$M_n:= \sum_{k=0}^{n-1}f(Y_{k+1})-Gf(Y_k),$$ be a bounded martingale since $(M_n)_{n\geq 0}$ is adapted to the natural filtration $(\mathcal F_n)_{n\geq 0}$ defined by $\mathcal F_n = \sigma(Y_0, \cdots, Y_n)$, $\forall n\geq 0$, and the increments $D_{k+1}:=f(Y_{k+1})-Gf(Y_k)$ has zero expectation, so $$\mathbb{E}(M_{n+1}|\mathcal{F}_n)=M_n+\mathbb E\left[f(Y_{n+1})-Gf(Y_n)\right|\mathcal F_n]=M_n+Gf(Y_n)-Gf(Y_n)=M_n.$$ By the strong law of large number for martingales with bounded increments (see e.g. \cite{Hall-Heyde-1980}, Theorem 2.18), it follows that
$$\lim_{n\to\infty}\bigl|\widehat\Pi^{x}_{n}f-\widehat\Pi^{x}_{n}Gf\bigr| = \lim_{n\to\infty} \frac{M_n}{n} =0,\quad\mathbb P-\text{a.s.}$$

We now obtain $$\lim_{t\to\infty}\Big| \hat \Pi_t^x(f)-\hat \Pi_{\lfloor t\rfloor}^x f\Big|\leq \lim_{t\to\infty}\Big| \Pi_t^x(f)-\hat \Pi_{\lfloor t\rfloor}^x Gf\Big| + \Big|\hat \Pi_{\lfloor t\rfloor}^x Gf-\hat \Pi_{\lfloor t\rfloor}^x f\Big|=0,$$ therefore by dominated convergence, $\lim_{t\to\infty}D\bigl(\Pi^{x}_{t},\widehat\Pi^{x}_{\lfloor t\rfloor}\bigr)=0$ almost surely and the two families share the same limit points, which must be invariant.
\end{proof}

\begin{proof}[\textbf{Proof of Proposition \ref{prop:extended-carre-du-champ-strong-law}}]
\textit{\textbf{(i)}} $\Rightarrow$ \textit{\textbf{(ii)}}: Fix $ x \in \mathcal M $ and, to shorten notation, set $ M_t = M^f_t(x) $ the local martingale induced by $f$ from (\ref{eq:extended-generator-implies}) with $ (\tau_n)_{n \geq 0} $ be its localizing sequence. Let $ M^n_t = M_{t \wedge \tau_n} $. Then, for $ t $ sufficiently large,
$$ \mathbb{E}\left[(M^n_t)^2\right] = \mathbb{E} \left[ \int_0^{t \wedge \tau_n} \Gamma^{\mathcal M}_e(f)(X^x_s) ds \right] \leq \int_0^t P_s \Gamma^{\mathcal M}_e(f)(x) ds \leq C_x t, $$
for some constant $ C_x $ using Fubini's theorem, the fact that $t\wedge \tau_n \leq t$, and $\Gamma_e^{\mathcal M}(f) \geq 0$. The sequence $ (M^n_t)_{n \geq 1} $ is then bounded in $ L^2 $, hence uniformly integrable. It therefore converges in $ L^1 $, as $ n \to \infty $, toward $ M_t $. Since $ (M^n_t)_{t \geq 0} $ is a martingale, by dominated convergence theorem, the $ L^1 $-convergence passes to the limit in the martingale property and $$ \mathbb{E}(M_t \mid \mathcal{F}_s) = \mathbb{E}(\lim_{n\to\infty }M^n_t \mid \mathcal{F}_s) = \lim_{n\to\infty } \mathbb{E}(M^n_t \mid \mathcal{F}_s) = \lim_{n\to\infty } M^n_s = M_s, \quad \text{for } 0 \leq s \leq t.$$

This proves that $ (M_t)_{t \geq 0} $ is a martingale. By Fatou's lemma, the previous inequality also implies that $\mathbb{E}[M_t^2] \leq C_x t,$ so that (\ref{eq:square-martingale}) is satisfied. \\

\textit{\textbf{(ii)}} $ \Rightarrow $ \textit{\textbf{(iii)}}: For all integers $ n $ and $ \varepsilon > 0 $, Doob's inequality for right-continuous martingales (see e.g. \cite{LeGall2016_Brownian}, Proposition 3.15) implies that
\begin{align*}
 \mathbb{P}_{x} \left( \sup_{2^n \leq t \leq 2^{n+1}} \frac{|M_t|}{t} \geq \varepsilon \right) &\leq \mathbb{P}_{x} \left( \sup_{t \leq 2^{n+1}} |M_t| \geq \varepsilon 2^n \right) \\
 &\leq \frac{\mathbb{E}_x\left[ \sup_{t \leq 2^{n+1}} |M_{t}|^2\right]}{\varepsilon^2 2^{2n}} \\
 &\leq \frac{4 C_x 2^{n+1}}{\varepsilon^2 2^{2n}} \\
 &\leq \frac{8 C_x}{\varepsilon^2 2^{n}},
\end{align*}
where the second inequality follows Markov's inequality, and the third one Doob’s inequality in $L^2$. Thus, $ \limsup_{t\to\infty }\frac{M_t}{t} = 0 $, $ \mathbb{P}_{x} $-almost surely by Borel–Cantelli lemma.
\end{proof}

We now are able to state the proof of Proposition \ref{prop:conditions-hypotheses-hold}.

\begin{proof}[\textbf{Proof of Proposition \ref{prop:conditions-hypotheses-hold}}]

\textit{\textbf{(i)}} follows directly from Lemma 7.26 and \textit{\textbf{(ii)}} by applying Corollary 4.23 in \cite{Benaim2022_Markov} with $\kappa=(1-e^{-a})\frac{b}{a}$, $\rho=e^{-a}$, so we only detail the proof of \textit{\textbf{(iii)}}.

Define $\nu_{n}=\tfrac1n\sum_{k=0}^{n-1}\delta_{X^{x}_{k}}$. By \textit{\textbf{(i)}}, it implies that $$P_{1} U \leq e^{-a}U+\frac{b}{a}.$$

Using \cite{Benaim2022_Markov}, Corollary 4.23, with $\kappa=\frac{b}{a}$ and $\rho=e^{-a}$, we have
\begin{equation}\label{eq:bounded-nu-tilde-W}
    \limsup_{n\to\infty}\nu_{n}\left(\sqrt{U}\right)\leq \frac{\frac{\sqrt{b}}{\sqrt a}}{1-\sqrt{e^{-a}}}.
\end{equation}
Writing $$\Pi^{x}_{n}\left(\sqrt{U}\right)=\frac{1}{n}\int_0^n \sqrt{U(X_s^x)}ds=\frac{1}{n}\sum_{k=0}^{n-1}\bigl[\Delta_{k+1}+\int_{0}^{1}P_{s}\sqrt{U\bigl(X^{x}_{k}}\bigr) \mathrm{d}s\bigr],$$

where $$\Delta_{k+1}=\int_0^1 \left(\sqrt{U(X^x_{k+s})}-P_s\sqrt{U(X_k^x)}\right)ds.$$

By Jensen's inequality, for all $s\geq 0$, we have $$P_s\left(\sqrt{U}\right)\leq \sqrt{P_s(U)}\leq e^{-\frac{\alpha s}{2}}\sqrt{U}+\frac{\sqrt{b}}{\sqrt{a}}\leq \sqrt{U}+\frac{\sqrt{b}}{\sqrt{a}},$$ so that $\int_0^1 P_s\left(\sqrt{U}\right) ds \leq \sqrt{U}+\frac{\sqrt{b}}{\sqrt{a}}$. Thus, $$\Pi^{x}_{n}\left(\sqrt{U}\right)\leq \frac{1}{n}\sum_{k=0}^{n-1}\Delta_{k+1}+\nu_n\left(\sqrt{U}\right)+\frac{\sqrt{b}}{\sqrt{a}}.$$

We now claim that $$\lim_{n\to\infty}\frac{1}{n}\sum_{k=0}^{n-1}\Delta_{k+1} =0,$$ $\mathbb P-$a.s. Combined with (\ref{eq:bounded-nu-tilde-W}), it implies that $$\limsup_{n\to\infty} \Pi^{x}_{n}\left(\sqrt{U}\right) \leq \left( 1+\frac{1}{1-\sqrt{e^{-a}}}\right)\frac{\sqrt{b}}{\sqrt{a}},$$ which concludes the proof. To prove the claim, using Fubini's theorem and Markov property, we observe that $$\mathbb{E}(\Delta_{k+1}\mid\mathcal{F}_{k})=\mathbb E\left(\int_0^1 \left(\sqrt{U(X^x_{k+s})}-P_s\sqrt{U(X_k^x)}\right)ds \mid \mathcal F_k\right)=0,$$ and $$\mathbb E(\Delta^2_{k+1} \mid \mathcal F_k)\leq \int_0^1\mathbb E\left[\left(\sqrt{U(X^x_{k+s})}-P_s\sqrt{U(X_k^x)}\right)^2\mid \mathcal F_k\right]ds,$$ by applying Cauchy-Schwarz and Fubini's theorem. Finally, recalling that the conditional variance is defined by $$\text{Var}(Z\mid \mathcal F)=\mathbb E\left[(Z - \mathbb E(Z\mid \mathcal F))^2\mid \mathcal F\right],$$ and with $\text{Var}(Z\mid \mathcal F)\leq \mathbb E(Z^2\mid \mathcal F)$, we have
\begin{align*}
\mathbb E[\Delta^2_{k+1} \mid \mathcal F_k] &\leq \int_0^1 \mathbb E\left[ \left(\sqrt{U(X^x_{k+s})}-\mathbb E\left[\sqrt{U(X_{k+s}^x)}\mid \mathcal F_k\right]\right)^2\mid \mathcal F_k \right]ds \\
&\leq \int_0^1\mathbb E[U(X_{k+s}^x)\mid \mathcal F_k]ds\\
&=\int_0^1 P_sU(X_k^x)ds\\
&\leq U(X_k^x)+\frac{\sqrt{b}}{\sqrt{a}},
\end{align*}
where $$\text{Var}\left(\sqrt{U(X^x_{k+s})}\mid \mathcal F_k\right)=\mathbb E\left[ \left(\sqrt{U(X^x_{k+s})}-\mathbb E\left[\sqrt{U(X_{k+s}^x)}\mid \mathcal F_k\right]\right)^2\mid \mathcal F_k \right],$$ while the last inequality follows (\ref{eq:hyp-3-b-implies-bounded-PW}) and the fact that $e^{-a}\geq 1-a$.

Using again previous bound together with (\ref{eq:hyp-3-b-implies-bounded-PW}), we have $$\sum_{k=0}^n \mathbb E[\Delta^2_{k+1}] = \sum_{k=0}^n \mathbb E[\mathbb E[\Delta^2_{k+1}|\mathcal F_k]]\leq \sum_{k=0}^n P_kU(x) + n\frac{b}{a}\leq \sum_{k=0}^n e^{-a k}U(x) + 2n\frac{b}{a}.$$ By the strong law of large number for discrete time martingale, it is a sufficient condition to prove the claim (see e.g. \cite{Benaim2022_Markov}, Theorem A.8 (iv)).\\

Tightness of $(\Pi_t^x)_{t\geq 0}$ follows Lemma 9.4 in \cite{Benaim2018_Persistence}, and any of its limit point lies in $\mathcal P_{\mathrm{inv}}( M)$ by Proposition \ref{prop:almost sure-limit-point-proba-inv}. 
\end{proof}

\subsection{Proof of Theorem \ref{thm:invasion-criterion}}\label{appendix-proof-invasion-rate}

\textit{\textbf{(i)}} Since $\mu$ is supposed to be an ergodic probability measure, it follows from Birkhoff ergodic Theorem that $\Pi_t^x\Rightarrow\mu$ for $\mu$ almost every $x$ and $\mathbb P_x-$almost surely. By Proposition \ref{prop:conditions-hypotheses-hold}\textit{\textbf{(iii)}}, using Lemma 9.4 in \cite{Benaim2018_Persistence} leads to $$\mu\left(\sqrt U\right)<\infty.$$ We now use assumption (\ref{eq:growth}) to conclude that $\mu\lambda_i<\infty$ for all $i\in I$ and $\forall \mu\in\mathcal P_{\mathrm{erg}}(M_0^I)$.\\

Let $i\in I$ and remark that $\lambda_i(x)=L(\log(x_i))$, so that the local martingale induced by $\log(x_i)$ verifies $$\frac{1}{t}\int_0^t \lambda_i(x_i^x(s))ds=\frac{\log(x_i^x(t))}{t}-\frac{\log(x_i^x(0))}{t}-\frac{M_t^{\log(x_i)}(x)}{t}.$$ In particular, $\Gamma_e^{i}(\log(x_i))\leq \mathrm{cst}$ so that $\log(x_i)$ satisfies the strong law and for $\mu\in \mathcal P_{\mathrm{erg}}(M_0^{I})$, $$\mu(\lambda_i)=\limsup_{t\to\infty}\frac{1}{t}\int_0^t \lambda_i(x_i^x(s))ds=\limsup_{t\to\infty}\frac{\log(x_i^x(t))}{t}, \quad \forall x\in M_+^{(i)}$$ 

If supp$(\mu)\not\subset M_0^{(i)}$, for $m\geq 1$, let $K_m=\left\{x\in M: x_i>\frac{1}{m} \right\}\cap [-m,m]$, so that $M_+^{(i)}=\cup_{m\geq 1}K_m$ and $M_0^{(i)}=\cap_{m\geq 1}K_m^c$. Since $\mu(M_0^{(i)})=0$, there exists $m^*\geq 1$ such that $\mu(K_m)\geq \frac{1}{2}$ and by Birkhoff ergodic Theorem, $(x_i^x(t))_{t\geq 0}$ visits $K_{m^*}$ infinitely often, for $\mu-$almost every $x$, $\mathbb P_x-$almost surely. 

Since $\log(x_i)$ is bounded on $K_{m^*}$, then $\mu\lambda_i=0$ on $K_{m^*}$, which extends to $M_+^{(i)}=\cup_{m\geq 1}K_m$. By taking the contrapositive, it follows that $\mu\lambda_i\neq 0$ implies that supp($\mu)\subset M_0^{(i)}$.\\

\textit{\textbf{(ii)}} We now suppose (\ref{eq:Hofbauer-cond}) and we show that conditions from Definition \ref{def:H-persistence} are verified. Let $h(u)=\log\left(\frac{1}{u}\right)$, $v:\mathbb R\to \mathbb R_+$ a $C^{\infty}$ function with bounded $v', \ v''$ such that $v(t)=t, \ \forall t\geq 1$, and $$V(x):=v\left(\sum_{i\in I} p_i h(x_i)\right).$$ 

Remark that $V$ coincides with $\sum_{i\in I} p_i h(x_i)$ on the subset $\{x\in M_+^I:\sum_{i\in I} p_i h(x_i)>1\}$. In particular, $V(x)\to\infty$ as $x_i\to 0$ for $i\in I$ since $h(u)\underset{u\to0}{\longrightarrow} \infty$ and $v(t)\underset{t\to\infty}{\longrightarrow}\infty$, so that $V$ is not defined on $M_0^I$.

From Proposition \ref{prop:domain-definition-c2}, it yields $$H|_{M_+^I}(x)=LV(x)=v'\left(\sum_{i\in I} p_ih(x_i)\right)\left(-\sum_{i\in I} p_i\lambda_i(x)\right)+\frac{1}{2}v''\left(\sum_{i\in I} p_ih(x_i)\right)\langle a(x)p, p\rangle_{\mathbb R^n},$$ for all $x\in M_+^I$.Then, $H|_{M_+^I}$ coincides $-\sum_{i\in I} p_i \lambda_i(x)$) on $\{x\in M_+^I:\sum_{i\in I} p_i h(x_i)>1\}$.

Moreover, $H$ extends continuously on $M_0^I$. Indeed, let $i\in I$ such that $x_i\to 0$. Since $h(x_i)\to\infty$, it yields $\sum_{i\in I} p_ih(x_i)>1$ and in particular, $H$ coincides with $-\sum_{i\in I} p_i \lambda_i(x)$ on $M_0^I$.\\

Then, $|H|\leq \mathrm{cst}\left( \sum_{i\in I} p_i |\lambda_i|+\sum_{i\in I} p_i^2 \right)$ which implies that $\frac{\sqrt{U}}{1+|H|}$ is proper since $$\frac{\sqrt{U}}{1+|H|}\geq C\cdot \frac{\sqrt{U}}{1+\sum_{i\in I}|F_i|}.$$ In particular, we also showed that $V\in \mathcal D_e^{2,I}$ by Proposition \ref{prop:domain-definition-c2} since $V$ is $C^2$ on $M_+$. Moreover, $V$ is a positive continuous function by construction and $$\Gamma_eV(x) = \sum_{i,j\in I} a_{ij}(x)p_ip_j\leq \mathrm{cst},$$ since $\Sigma_i^j$ is bounded, so $V$ satisfies the strong law by Corollary \ref{cor:hyp-3-implies}.

Then, for any $\mu\in\mathcal P_{erg}(M_0^I)$, condition (\ref{eq:Hofbauer-cond}) implies that $\mu H = -\sum_{i\in I} p_i\mu\lambda_i <0$ since $\mu$ charges $M_0^I$ and $\{(X_t^x)_{t\geq 0}: x\in M_+^I\}$ is $H-$persistent. \qed

\bibliographystyle{alpha}
\bibliography{biblio.bib}
\end{document}